\documentclass[ejs]{imsart}

\RequirePackage[OT1]{fontenc}
\RequirePackage{amsthm,amsmath}
\RequirePackage[numbers]{natbib}
\RequirePackage[colorlinks,citecolor=blue,urlcolor=blue]{hyperref}

\usepackage{amssymb}
\usepackage[T1]{fontenc}
\usepackage[latin1]{inputenc}
\usepackage[english]{babel}
\usepackage{amsfonts}
\usepackage{enumitem}
\usepackage{amsmath}
\usepackage{amsthm}
\usepackage{booktabs}
\usepackage[normalem]{ulem}
\usepackage{rotating}
\usepackage{amssymb}
\usepackage{tabularx,ragged2e,booktabs,caption}
\usepackage{geometry}
\usepackage{multirow}
\usepackage{subcaption}
\usepackage{stmaryrd}

\startlocaldefs
\numberwithin{equation}{section}
\theoremstyle{plain}

\def\R{\mathbb{R}}

\def\N{\mathbb{N}}
\def\P{\mathbb{P}}
\def\E{\mathbb{E}}
\def\L{\mathbb{L}}

\def\R{\mathbb{R}}

\def\Z{\mathbb{Z}}

\def\1{\mbox{I\hspace{-.6em}1}} 

\def\cov{\mbox{Cov}\,}

\def\1{\mbox{\hspace{.2em}I\hspace{-.6em}1}} 

\def\limiteasn{\renewcommand{\arraystretch}{0.5}
    \begin{array}[t]{c}\stackrel{a.s.}{\longrightarrow} \\
        {\scriptstyle
            n\rightarrow+\infty}\end{array}\renewcommand{\arraystretch}{1}}

\def\limiteloin{\renewcommand{\arraystretch}{0.5}
    \begin{array}[t]{c}\stackrel{{\cal L}}{\longrightarrow} \\
        {\scriptstyle
            n\rightarrow+\infty}\end{array}\renewcommand{\arraystretch}{1}}

\DeclareMathOperator{\argmin}{argmin}
\DeclareMathOperator{\argmax}{argmax}
\theoremstyle{plain}
\newtheorem{theo}{Theorem}[section]
\newtheorem{lemma}{Lemma}
\newtheorem{proposition}{Proposition}
\numberwithin{equation}{section}

\theoremstyle{definition}
\newtheorem{remark}{Remark}

\newcommand*\interior[1]{\overset{\mathsf{o}}{#1}}

\def\limiten{\renewcommand{\arraystretch}{0.5}
    \begin{array}[t]{c}
        \stackrel{}{\longrightarrow} \\
        {\scriptstyle n \rightarrow\infty}
    \end{array}\renewcommand{\arraystretch}{1}}

\def\limiteproban{\renewcommand{\arraystretch}{0.5}
    \begin{array}[t]{c}
        \stackrel{a.s.}{\longrightarrow} \\
        {\scriptstyle n \rightarrow\infty}
    \end{array}\renewcommand{\arraystretch}{1}}

\def\limiteproban{\renewcommand{\arraystretch}{0.5}
    \begin{array}[t]{c}
        \stackrel{{\mathcal P}}{\longrightarrow} \\
        {\scriptstyle n \rightarrow\infty}
    \end{array}\renewcommand{\arraystretch}{1}}

\def\limitesur{\renewcommand{\arraystretch}{0.5}
    \begin{array}[t]{c}
        \stackrel{{a.s.}}{\longrightarrow} \\
        {\scriptstyle n \rightarrow\infty}
    \end{array}\renewcommand{\arraystretch}{1}}

\endlocaldefs

\begin{document}

    \begin{frontmatter}
        \title{Consistent model selection criteria and goodness-of-fit test  for affine causal processes}
        \runtitle{Consistent model selection criteria and goodness-of-fit test  for affine causal processes}

\begin{aug}
\author{\fnms{Jean-Marc} \snm{Bardet}{}}
\and
\author{\fnms{Kare} \snm{Kamila}\thanksref{t1}\ead[label=e1]{Jean-Marc.Bardet@univ-paris1.fr and kamilakare@gmail.com}}

\address{S.A.M.M., Université Paris 1, Panthéon-Sorbonne,\\
    90, rue de Tolbiac, 75634, Paris, France\\
\printead{e1}}

\author{\fnms{William} \snm{Kengne}
\ead[label=e2]{william.kengne@gmail.com}}

\address{THEMA, Université de Cergy-Pontoise, FRANCE.\\
\printead{e2}}

\thankstext{t1}{Developed within the European Union's Horizon 2020 research
and innovation programme under the Marie Sklodowska-Curie grant agreement No 754362.}
\thankstext{t2}{Developed within the ANR BREAKRISK : ANR-17-CE26-0001-01.}
\runauthor{Bardet et al.}
\end{aug}
\begin{abstract}
%
 This paper studies the model selection problem in a large class of causal time series models, which includes both the ARMA or AR($\infty$) processes, as well as the
 GARCH or ARCH($\infty$), APARCH, ARMA-GARCH and many others processes.
 To tackle this issue, we consider a  penalized contrast based on the quasi-likelihood of the model.
 We provide sufficient conditions for the penalty term to ensure the consistency of the proposed procedure as well as the consistency and the asymptotic normality of the
 quasi-maximum likelihood estimator of the chosen model.
 It appears from these conditions that the Bayesian Information Criterion (BIC) does not always guarantee the consistency.
  We also propose a  tool for diagnosing the goodness-of-fit of the chosen model based on the portmanteau Test.
  Numerical simulations and an illustrative example on the FTSE index are performed to highlight the obtained asymptotic results, including a numerical evidence of the non consistency of the usual BIC penalty for order selection of an  AR$(p)$ models with ARCH($\infty$) errors.
\end{abstract}


\begin{keyword}[class=MSC]
    \kwd[Primary ]{60K35}
    \kwd{60K35}
    \kwd[; secondary ]{60K35}
\end{keyword}

\begin{keyword}
    \kwd{model selection}
    \kwd{affine causal processes}
    \kwd{consistency}
    \kwd{BIC}
    \kwd{Portmanteau Test}
\end{keyword}

\end{frontmatter}

\section{Introduction}

Model selection is an important tool for statisticians and all those who process data. This issue has received considerable attention in the recent literature.
There are several model selection procedures, the main ones are : cross validation and penalized contrast based.

 \medskip

\noindent The cross validation (\cite{stone}, \cite{allen}) consists in splitting the data into learning sample, which will be used for computing estimators of the parameters and the test sample which allows to assess these estimators by evaluate their risks.

\medskip

\noindent The procedures using penalized objective function search for a model, minimizing a trade-off between a  sum of an empirical risk (for instance least squares, $-2\times$log-likelihood), which indicates how well the model fits the data, and a measure of model's complexity so-called a penalty.
\\
The idea of penalizing  dates back to the 1970s with the works of \cite{mallows} and \cite{akaike}. By using the ordinary least squares in regression framework, Mallows obtained the $C_p$ criterion. Meanwhile, Akaike derived AIC for density estimation using log-likelihood contrast. A few years later, following Akaike,  \cite{schwarz} proposed an alternative approach to density estimation  and derived the Bayesian Information Criteria (BIC).
The penalty term of these criteria is proportional to the dimension of the model. In the recent decades, different approaches of penalization have emerged such
as the $\L^2$ norm for the Ridge penalisation \cite{hoer}, the $\L^1$ norm  used by \cite{tibshirani} that provides the LASSO procedure and the elastic-net that
mixes the $\L^1$ and $\L^2$ norms \cite{zou}.

Model selection procedures can have two different objectives: \textit{consistency} and \textit{efficiency}. A procedure is said to be consistent if given a family of models, including the "true model", the probability of choosing the correct model approaches one as the sample size tends to infinity. On the other hand, a procedure is efficient when its risk is asymptotically equivalent to the risk of the oracle. In this work, we are interested to construct a consistent procedure for the general class of times series known as \textit{affine causal processes}.


This class of affine causal time series can be defined as follows. Let $\R^\infty$ be the space of sequences of real numbers with a finite number of non zero,
if $M$, $f$ : $\R^\infty \to \R$ are two measurable functions, then an affine causal class is\\
~\\
\textbf{Class} $\mathcal{AC}(M,f):$ A process $X=(X_t)_{t\in \Z}$ belongs to $\mathcal{AC}(M,f)$ if it satisfies:
\begin{equation}
X_t=M\big((X_{t-i})_{i\in \mathbb{N}^*}\big)\, \xi_t+f\big((X_{t-i})_{i\in \mathbb{N}^*}\big) \;\; \mbox{for any}~ t\in \Z;
\label{eq:serie2}
\end{equation}
where $(\xi)_{t\in \Z}$ is a sequence  of zero-mean independent identically distributed random vectors (i.i.d.r.v) satisfying $\mathbb{E}(|\xi_0|^r)<\infty $ for some
 $r\geq 2$ and $\E[\xi_0^2]=1$.

 \medskip

 \noindent For instance,
\begin{itemize}
\item if $M\big((X_{t-i})_{i\in \mathbb{N}^*}\big)=\sigma$ and $f\big((X_{t-i})_{i\in \mathbb{N}^*}\big)=\phi_1X_{t-1}+\cdots+\phi_pX_{t-p}$,
  then $(X_t)_{t \in \Z}$ is an AR$(p)$ process;
\item if $M\big((X_{t-i})_{i\in \mathbb{N}^*}\big)=\sqrt{a_0+a_1X^2_{t-1}+\cdots+a_pX^2_{t-p}}$ and $f\big((X_{t-i})_{i\in \mathbb{N}^*}\big)=0$,
then $(X_t)_{t \in \Z}$ is an ARCH$(p)$ process.
\end{itemize}
Note that, numerous classical  time series models such as  ARMA($p,q$), GARCH($p,q$), ARMA($p,q$)-GARCH($p,q$) (see \cite{ding} and \cite{ling})
or APARCH$(\delta,p, q)$ processes (see \cite{ding}) belongs to $\mathcal{AC}(M,f)$.
 The existence of stationary and ergodic solutions of this class has been studied in \cite{dou} and \cite{barW}.

  \medskip

 We consider a trajectory $(X_1,\ldots,X_n)$ of a stationary affine causal process $\mathcal{AC}(M^*,f^*)$, where $M^*$ and $f^*$ are unknown.
We also consider a finite set ${\cal M}$ of parametric models $m$, which are affine causal time series.
We assume that the "true" model $m^*$ corresponds to $M^*$ and $f^*$.
 The aim is to obtain an estimator $\widehat m$ of $m^*$ and testing the goodness-of-fit of the chosen model.

 \medskip

There already exist several important contributions devoted to the model selection for time series ;
  we refer to the book of \cite{mcQ} and the references therein for an overview on this topic. \\
  As we have pointed above, two properties are often used to evaluate a quality of a model selection procedure : consistency and efficiency.
  The first measure is often used when the true model is included in the collection of  model's candidate ; otherwise, efficiency is the well-defined property.
  In many research in this framework, the main goal is to develop a procedure that fulfills one of these properties.
  So, in some classical linear time series models, the consistency of the BIC procedure has been established, see for instance \cite{Hannan1980} or
  \cite{Tsay1984} ; and the asymptotic efficiency of the AIC has been proved, see, among others, \cite{Shibata1980}, \cite{Hurvich1989} for a corrected version
  of AIC for small samples, \cite{Ing2005}, \cite{Ing2007}, \cite{Ing2012} for the case of infinite order autoregressive model.
  \cite{Shi2002} propose the (consistent) residual information criteria (RIC) for regression model (including regression models with ARMA errors) selection.
  In the framework of nonlinear threshold models, \cite{Kapetanios2001} proved consistency results of a large class of information criteria, whereas \cite{Gao2004} focussed on
  cross-validation type procedure for model selection in a class of semiparametric time series regression model.
  Let us recall that, the time series model selection literature is very extensive and still growing ; we refer to the monograph of \cite{Rao2001}, which provided an excellent
  summary of existing model selection procedure, including the case of time series models as well as the recent review paper of \cite{Ding2018}.

 \medskip

  The adaptive lasso, introduced by \cite{Zou2006} for variable selection in linear  regression models  has been extended by \cite{Ren2010} to vector autoregressive
  models, \cite{Kock2016} carried out this procedure in stationary and nonstationary autoregressive models ; the oracle efficient is established.
  \cite{Lerasle2011} considers model selection for density estimation under mixing conditions and derived oracle inequalities of the slope heuristic procedure (\cite{Birge2007} or \cite{arlot}) ;
  whereas \cite{Alquier2012} develop oracle inequalities for model selection for weakly dependent time series forecasting.
  Recently, \cite{Shao2017} have considered the model selection for ARMA time series with trend, and proved the consistency of BIC for the detrended residual sequence, while
  \cite{Arkoun2018} developed oracle inequalities of sequential model selection method for nonparametric autoregression.
  \cite{Hsu2019} pointed out that most existing model selection procedure cannot simultaneously enjoy consistency and (asymptotic) efficiency. They propose a
  misspecification-resistant information criterion that can achieve consistency and asymptotic efficiency for prediction using model selection. 
~\\

 In this paper, we focus on the class of models (\ref{eq:serie2}), and addressed the following questions :
  \begin{enumerate}
    \item What regularity conditions are sufficient to build a consistent model selection procedure? Does the classic criterion such as BIC, still have consistent property
    for choosing a model among the collections $\mathcal{M}$?
    \item How can we test the goodness-of-fit of the chosen model?
  \end{enumerate}

 These questions have not yet been answered for the class of models and the framework considered here, in particular in case of infinite memory processes.
  This new contribution provides theoretical and numerical response of these issues.

 \medskip
 (i) The estimator $\widehat m$ of $m^*$
 is chosen  by minimizing a penalized criterion $\widehat{C}(m)=-2\widehat{L}_n(m)+ |m|\,\kappa_n$, where $\widehat{L}_n(m)$ is a Gaussian quasi-log-likelihood of the model $m$,
 $|m|$ is the number of estimated parameters of the model $m$ and $\kappa_n$ is a non-decreasing sequence of real numbers (see more details in Section \ref{Def}). Note that, in the cases $\kappa_n=2$ or $\kappa_n=\log n$ we respectively consider the usual AIC and BIC criteria.
 We provide sufficient conditions (essentially depending on the decreasing of the Lipschitz coefficients of the functions $f$ and $M$) for obtaining consistency of the model selection
 procedure.
 We also theoretically and numerically exhibit an example of order selection  (weak AR$(p)$ processes with ARCH($\infty$) errors) such that the consistency of the
 classical BIC penalty is not ensured.

 \medskip
 (ii) We provide an asymptotic goodness-of-fit test for the selected model that is very simple to be used (with the usual Chi-square distribution limit), which successively
completes the model selection procedure. Numerical applications show the accuracy of this test under the null hypothesis as well as an efficient test power under an
alternative hypothesis.
Note that, similar test has been proposed by \cite{li2} under the Gaussian assumption on the observations,  whereas \cite{Ling1997} focused for multivariate time series
with multivariate ARCH-type errors.
Also, \cite{Francq2008} proposed a portmanteau test statistic based on generalized inverses and $\{2 \}$-inverses for diagnostic checking in the class of model (\ref{eq:serie2}).
Unlike these authors, we apply the test to a model obtained from a model selection procedure. 

 \medskip



~\\
The paper is organized as follows. Some definitions, notations and assumptions are described in Section \ref{Def}. The consistency of the criteria and the asymptotic normality of the post-model-selection estimator are studied in Section \ref{Asympto}. In Section \ref{Examples}, the examples of $AR(\infty)$, $ARCH(\infty)$, $APARCH(\delta,p,q)$ and ARMA$(p,q)$-GARCH$(p',q')$ processes are detailed. The goodness-of-fit test is presented in Section \ref{Test}. Finally, numerical results are presented in Section \ref{Numeric} and Section \ref{Proofs} contains the proofs.

\section{Definitions and Assumptions}\label{Def}
%
Let us introduce some definitions and assumptions in order to facilitate the presentation.
\subsection{Notation and assumptions}
In the sequel, we will consider a subset $\Theta$  of $ \mathbb{R}^{d}$ ($d \in \N$). We will use the following norms:
\begin{itemize}
\item $\|.\|$  denotes the usual Euclidean norm on $\R^\nu$, with $\nu\geq 1$;
\item if $X$ is $\mathbb{R}^\nu$-random variable with $r\ge 1$ order moment, we set $\|X\|_r=\big (\E(\|X\|^r\big)^{1/r}$;
\item for any set $\Theta \subseteq\mathbb{R}^d$ and for any $g: \Theta \to \R^{d'}$, $d'\geq 1$, denote $\|g\|_{\Theta}=\underset {\theta \in \Theta}{\sup} \big \{ \|g(\theta)\|\big \}$.
\end{itemize}
In the introduction, to be more concise, we have presented the problem of time series model selection in a very general form. In reality, we will limit our field of study a little bit by considering a semi-parametric framework.
Hence, let $(f_{\theta})_{\theta \in \Theta}$ and $(M_{\theta})_{\theta \in \Theta}$ be two families of known functions such as for any $\theta \in \Theta$, both $f_{\theta}, M_{\theta}$  with real values defined on $\R^\infty$.

 \medskip

\noindent We begin by giving a condition on $f_{\theta}$ and  $M_{\theta}$ which ensure the existence of a $r$-order moment, stationary and ergodic time series belonging
to $\mathcal{AC}(M_\theta,f_\theta)$. This condition, initially  obtained in \cite{dou}, is written in terms of Lipschitz coefficients of both these functions. Hence, for $\Psi_{\theta}=f_{\theta}$ or  $M_{\theta}$, define:

\medskip

\noindent \textbf{Assumption A}$(\Psi_{\theta},\Theta)$: {\it Assume that $\|\Psi_{\theta}(0)\|_\Theta < \infty$ and there exists a sequence of non-negative real numbers $\big(\alpha_k(\Psi_{\theta},\Theta)\big)_{k\ge 1}$ such that $\sum_{k=1}^{\infty}\alpha_k(\Psi_{\theta},\Theta)< \infty$ satisfying:
\[
\|\Psi_{\theta}(x)-\Psi_{\theta}(y)\|_\Theta \le \sum_{k=1}^{\infty}\alpha_k(\Psi_{\theta},\Theta)|x_k-y_k| \; for \; all\; x,y \in \mathbb{R}^\mathbb{\infty}.
\]}

\medskip

\noindent Now for $r\geq 1$, where $\|\xi_0\|_r<\infty$, define:
\begin{multline}
\Theta(r)=\Big \{\theta  \in \R^d,~ A(f_{\theta},\{\theta\})\; \textnormal{and}\; A(M_{\theta},\{\theta\})\; \textnormal{hold with} \\
 \sum_{k=1}^{\infty} \alpha_k(f_{\theta},\{\theta\}) +\|\xi_0\|_r \, \sum_{k=1}^{\infty} \alpha_k(M_{\theta},\{\theta\}) < 1 \Big \}.
\end{multline}
Then, for any $\theta \in \Theta(r)$, there exists a stationary and ergodic solution with $r$-order moment belonging to $\mathcal{AC}(M_\theta,f_\theta)$.
 (see \cite{dou} and \cite{barW}).

\medskip
\subsection{The framework}
Let us start with an example to better understand the framework and the approach of model selection we will follow.\\
~\\
{\bf Example:} Assume that the observed trajectory $(X_1,\ldots,X_n)$ is generated from an AR$(2)$ process and we would like to identify this family of process and its order. Then, we consider the collection $\cal{M}$ of ARMA$(p,q)$ and GARCH$(p',q')$ processes for $0\leq p,q,p',q'\leq 9$ and we would like to chose in this family a "best" model for fitting $(X_1,\ldots,X_n)$. Note that there is $200$ possible models and we expect to recognize the AR$(2)$ as the selected model, at least when $n$ is large enough.

\medskip

We begin with the following property that allow to enlarge the family of models by extending the dimension $d$ of the parameter $\theta$:
\begin{proposition}\label{prop0}
Let $d_1,d_2 \in \N$, $\Theta_1 \subset \mathbb{R}^{d_1}$ and $\Theta_2 \subset \mathbb{R}^{d_2}$, and for $i=1,2$, define $f_{\theta_i}^{(i)}, M_{\theta_i}^{(i)} :\R^\infty \to \R$ and  for $\theta_i\in \Theta_i$. Then there exist $\max(d_1,d_2)\leq d\leq d_1+d_2$, $\Theta \subset \R^d$, and a family of functions $f_{\theta}:\R^\infty \to \R$ and $M_{\theta}:\R^\infty \to [0,\infty)$ with $\theta\in \Theta$,
such that for any $\theta_1\in \Theta_1$ and $\theta_2\in \Theta_2$, there exists $\theta\in \Theta$ satisfying
\[
\mathcal{AC}\big (M_{\theta_1}^{(1)},f_{\theta_1}^{(1)}\big)\bigcup \mathcal{AC}\big(M_{\theta_2}^{(2)},f_{\theta_2}^{(2)}\big) \subset \mathcal{AC}\big(M_{\theta},f_{\theta}\big).
\]
\end{proposition}
\noindent The proof of this proposition, as well as the other proofs, can be found in Section \ref{Proofs}.
This proposition says that it is always possible to embed two parametric causal affine models in a larger one. Hence, for instance, we can consider as well AR processes and ARCH processes in a unique representation, {\it i.e.}
\begin{multline*}
\left \{ \begin{array}{ll}
AR&\left\{\begin{array}{l}
M_{\theta_1}^{(1)}\big((X_{t-i})_{i\in \mathbb{N}^*}\big)=\sigma \\
f_{\theta_1}^{(1)}\big((X_{t-i})_{i\in \mathbb{N}^*}\big)=\phi_1X_{t-1}+\cdots+\phi_pX_{t-p}
\end{array}\right .
\\
 \\
ARCH&\left\{\begin{array}{l}
M_{\theta_2}^{(2)}\big((X_{t-i})_{i\in \mathbb{N}^*}\big)=\sqrt{a_0+a_1X^2_{t-1}+\cdots+a_qX^2_{t-q}} \\
f_{\theta_2}^{(2)}\big((X_{t-i})_{i\in \mathbb{N}^*}\big)=0
\end{array}\right .
\end{array}  \right .\\
\Longrightarrow \left\{
\begin{array}{l}
M_{\theta}\big((X_{t-i})_{i\in \mathbb{N}^*}\big)=\sqrt{\theta_0+\theta_1X^2_{t-1}+\cdots+\theta_qX^2_{t-q}} \\
f_{\theta}\big((X_{t-i})_{i\in \mathbb{N}^*}\big)=\theta_{q+1}X_{t-1}+\cdots+\theta_{q+p}X_{t-p}
\end{array} \right . .
\end{multline*}
From now and in all the sequel, we fix $d \in \N^*$, and the family of functions $f_{\theta},M_{\theta}:\R^\infty \to \R$
for $\theta\in \Theta \subset \Theta(r) \subset \R^d$.

\medskip

Let $(X_1,\ldots, X_n)$ be an observed trajectory of an affine causal process $X$ belonging to $\mathcal{AC}(M_{\theta^*},f_{\theta^*})$, where $\theta^*$ is an unknown vector of $\Theta$, and therefore:
\begin{equation}
X_t=M_{\theta^*}\big((X_{t-i})_{i\in \mathbb{N}^*}\big)\, \xi_t+f_{\theta^*}\big((X_{t-i})_{i\in \mathbb{N}^*}\big) \;\; \mbox{for any}~ t\in \Z.
\label{eq:serie}
\end{equation}
In the sequel, we will consider  several models, which all are particular cases of $\mathcal{AC}(M_{\theta},f_{\theta})$ with $\theta \in \Theta \subset \R^d$. More precisely define:
\begin{itemize}
\item a model $m$ as a subset of $\{1,\ldots,d\}$ and denote $|m|=\#(m)$;
\item $\Theta(m)=\big \{ (\theta_i)_{1\leq i \leq d}\in \R^d,~\theta_i=0~\mbox{if $i\notin m$} \big \}\cap \Theta$;
\item ${\cal M}$ as a family of models, {\it i.e.} ${\cal M}\subset {\cal P} \big ( \{1,\ldots,d\}\big )$.
\end{itemize}
Finally, for all $m\in {\cal M}$, $m \in \mathcal{AC}(M_{\theta},f_{\theta})$ when $\theta \in \Theta(m)$ and denote $m^*$ the "true" model.
We could as well consider hierarchical or exhaustive families of models. \\
~\\
{\bf Example:} From the previous example, we can consider:\\
$\bullet$ a family ${\cal M}_1$ of models $m_1$ such as ${\cal M}_1=\big \{\{1\},\{1,2\},\ldots,\{1,\ldots,q+1\} \big \}$: this family is the hierarchical one of ARCH processes with orders varying from $0$ to $q$. \\
$\bullet$ a family ${\cal M}_2$ of models $m_2$ such as ${\cal M}_2= {\cal P} \big ( \{1,\ldots,p+q+1\}\big )$: this family is the exhaustive one and contains as well the AR$(2)$ process $X_t=\phi_2 X_{t-2}+ \theta_0 \, \xi_t$ as the process  $X_t=\phi_1X_{t-1}+\phi_3 X_{t-3}+  \xi_t \,\sqrt{ \theta_0 +a_2 X_{t-2}^2 }$.\\
~\\
To establish the consistency of the selected model, we will need to assume that the "true" model $m^*$ with the parameter $\theta^*$, is included in the model family ${\cal M}$.
\subsection{The special case of NLARCH$(\infty)$ processes}
As in \cite{barW}, in the special case of NLARCH$(\infty)$ processes, including for instance GARCH$(p,q)$ or ARCH$(\infty)$ processes, a particular treatment can be realized for obtaining sharper results than using the previous framework. In such case, define the class:\\
~\\
\textbf{Class} $\widetilde{\mathcal{AC}}(\widetilde H_\theta)$: A process $X=(X_t)_{t\in \Z}$ belongs to $\widetilde{\mathcal{AC}}(\widetilde H_\theta)$ if it satisfies:
\begin{equation}
X_t=\xi_t  \, \sqrt {\widetilde H_\theta \big((X^2_{t-i})_{i\in \mathbb{N}^*}\big)}\;\; \mbox{for any}~ t\in \Z.
\label{eq:serie3}
\end{equation}
Therefore, if $ M^2_\theta \big((X_{t-i})_{i\in \mathbb{N}^*}\big)=H_\theta \big((X_{t-i})_{i\in \mathbb{N}^*}\big)=\widetilde H_\theta \big((X^2_{t-i})_{i\in \mathbb{N}^*}\big)$
then, $\widetilde{\mathcal{AC}}(\widetilde H_\theta)={\mathcal{AC}}(M_\theta,0)$.
In case of the class $\widetilde{\mathcal{AC}}(\widetilde H_\theta)$, we will use the  assumption $A(\widetilde H_{\theta},\Theta)$.
By this way, we will obtain a new set of stationary solutions. For $r\geq 2$ define:
\begin{equation}
\widetilde {\Theta}(r)=\Big \{\theta  \in \R^d,~ {A}(\widetilde H_{\theta},\{\theta\})~ \mbox{holds with}~
 \big (\|\xi_0\|_r \big )^2 \, \sum_{k=1}^{\infty} \alpha_k(\widetilde H_{\theta},\{\theta\}) < 1 \Big \}.
\end{equation}
Then, for $\theta \in {\Theta}(r)$, a process $(X_t)_{t\in \Z}$ belonging to the class $\widetilde{\mathcal{AC}}(\widetilde H_\theta)$ is  stationary ergodic and satisfies $\|X_0\|_r<\infty$.
\subsection{The Gaussian quasi-maximum likelihood estimation and the model selection criterion} \label{qmle_msc}
In the sequel, for a model $m \in {\cal M}$, a family of models of $\mathcal{AC}(M_{\theta},f_{\theta})$ with $\theta \in \Theta \subset \R^d$, where $\theta \to M_{\theta}$ and $\theta \to f_{\theta}$ are two fixed functions, we are going to consider Gaussian quasi-maximum likelihood estimators (QMLE) of $\theta$ for each specific model $m$.

\medskip

\noindent This approach as semi-parametric estimation has been successively introduced for GARCH$(p,q)$ processes in \cite{Jeantheau1998} where its consistency is also proved, and the asymptotic normality of this estimator has been established in \cite{Berkes2003} and \cite{Francq2004}. In \cite{barW}, those results have been extended to affine causal processes, and an extension to Laplacian QMLE has been also proposed in \cite{barY}. \\
\noindent The Gaussian QMLE is derived from the conditional (with respect to the filtration $\sigma \big \{(X_{t})_{t\leq 0} \big \}$) log-likelihood of $(X_1,\ldots,X_n)$ when $(\xi_t)$
is supposed to be a Gaussian standard white noise.
Due to the linearity of a causal affine process, we deduce that this conditional log-likelihood (up to an additional constant) $L_n$ is defined for all $\theta \in \Theta$ by:
\begin{equation}
L_n(\theta):=-\frac{1}{2}\sum_{t=1}^n q_t (\theta) ~ , ~ \textnormal{with} \; q_t (\theta):=\frac{(X_t - f_{\theta}^t)^2}{H_{\theta}^t} + \log(H_{\theta}^t)
\label{eq:eq1}
\end{equation}
where $f_{\theta}^t:=f_{\theta}(X_{t-1},X_{t-2},\cdots)$, $M_{\theta}^t:=M_{\theta}(X_{t-1},X_{t-2},\cdots)$ and $H_{\theta}^t=\big (M_{\theta}^t\big )^2$.
Since $L_n(\theta)$ depends on $(X_t)_{t\le 0}$ that are unknown, the idea of the  quasi log-likelihood is to replace $q_t (\theta)$ by an approximation $\widehat{q}_t(\theta)$
and to compute $\widehat{\theta}$ as in equation \eqref{eq:qmle} even if the white noise is not Gaussian.
Hence, the conditional quasi log-likelihood (up to an additional constant) is given for all $\theta \in \Theta$ by
\begin{multline}
\widehat{L}_n(\theta):=-\frac{1}{2}\, \sum_{t=1}^n \widehat{q}_t (\theta) ~ , ~ \textnormal{with} \; \widehat{q}_t (\theta):=\frac{(X_t - \widehat{f}_{\theta}^t)^2}{\widehat{H}_{\theta}^t} + \log(\widehat{H}_{\theta}^t) \\
\mbox{where}
\displaystyle \left \{
\begin{array}{lcl}
\widehat{f}_{\theta}^t&:=&f_{\theta}(X_{t-1},X_{t-2},\cdots,X_1,u)\\
\widehat{M}_{\theta}^t&:=&M_{\theta}(X_{t-1},X_{t-2},\cdots,X_1,u)  \\
\widehat{H}_{\theta}^t&:=&(\widehat{M}_{\theta}^t)^2
\end{array} \right .
\end{multline}
for any deterministic sequence $u=(u_n)$ with finitely many non-zero values ($u =0$ is very often chosen without loss of generality).

\medskip

However, the definitions of the conditional log-likelihood and quasi log-likelihood require that their denominators do not vanish.
Hence, we will suppose in the sequel that the lower bound of $H_{\theta}(\cdot) = \big(M_{\theta}(\cdot) \big)^2$ (which is reached since $\Theta$ is compact) is strictly positive:

\medskip

\noindent \textbf{Assumption D}$(\Theta)$: {\it $\exists \underline{h}>0$ such that $\underset{\theta \in \Theta}{\inf}(H_{\theta}(x)) \ge \underline{h}$ for all $x\in \mathbb{R}^{\infty}$.}

\medskip

\noindent Finally, under this assumption, for each specific model $m\in {\cal M}$, we define the Gaussian QMLE   $\widehat{\theta}(m)$ as
\begin{equation}
\widehat{\theta}(m)= \underset{\theta \in \Theta(m)}{\argmax} \;\widehat{L}_n(\theta).
\label{eq:qmle}
\end{equation}
To select the "best" model $m\in {\cal M}$, we chose a penalized contrast $\widehat{C}(m)$ ensuring a trade-off between $-2$ times the maximized quasi log-likelihood, which decreases with the size of the model, and a penalty increasing with the size of the model. Therefore, the choice of the "best" model $\widehat{m}$ among the estimated can be performed by minimizing the following criteria
\begin{equation}
\widehat{m}=  \underset{m \in \mathcal{M}}{\argmin} \;\widehat{C}(m)\quad\textnormal{with}\quad
\widehat{C}(m)=-2\widehat{L}_n\big(\widehat{\theta}(m)\big)+ |m| \,\kappa_n,
\label{eq:cri}
\end{equation}
where
\begin{itemize}
\item $(\kappa_n)_n$ an increasing sequence depending on the number of observations $n$.
\item $|m|$ denotes the dimension of the model $m$, {\it i.e.} the cardinal of $m$, subset of $\{1,\ldots,d\}$, which is also the number of estimated components of $\theta$ (the others are fixed to zero).
\end{itemize}
The consistency of the criterion  $\widehat{C}$, {\it i.e.}
\begin{equation}
\mathbb{P}(\widehat{m}=m^*)\limiten 1 ;
\label{eq:cons}
\end{equation}
will be established after showing that both of following probabilities are zero:
\begin{itemize}
\item the asymptotic probability of selecting a larger model containing the true model (overfitting case);
\item the asymptotic probability of selecting a false model that is a model not containing $m^*$.
\end{itemize}

\section{Asymptotic results}\label{Asympto}
\label{sec:3}
\subsection{Assumptions required for the asymptotic study}
%
%
The following classical assumption ensures the identifiability of the model considered.

\medskip

\noindent \textbf{Assumption Id}$(\Theta)$: {\it  For all $\theta, \, \theta' \in \Theta$,
 $(f_{\theta}^0=f_{\theta'}^0 \; and\;  M_{\theta}^0=M_{\theta'}^0 ) ~ a.s. \implies \theta=\theta'. $ }

 \medskip

 \noindent Another required assumption concerns the differentiability of $\Psi_{\theta}=f_{\theta}$ or  $M_{\theta}$ on $\Theta$. This type of assumption has already been considered in order to apply the QMLE procedure (see \cite{barW}, \cite{strau}, \cite{white}). First, the following Assumption Var$(\Theta)$ provides the invertibility of the "Fisher's information matrix" of $X$ and is important to prove the asymptotic normality of the QMLE.

 \medskip

\noindent \textbf{Assumption Var}: {\it $\big (\sum_{i=1}^d \alpha_i \, \frac {\partial f_{\theta}^0} {\partial \theta^{(i)}} =0~\implies ~\forall i=1,\ldots,d, ~\alpha_i=0~a.s\big )$ or $\big (\sum_{i=1}^d \alpha_i \, \frac {\partial H_{\theta}^0} {\partial \theta^{(i)}} =0~\implies ~\forall i=1,\ldots,d, ~\alpha_i=0~a.s\big )$.}

\medskip

 \noindent Moreover,  one of the following technical assumption is required to establish the consistency of the model selection procedure.

\medskip

\noindent \textbf{Assumption $\boldsymbol{K}(\Theta)$}: {\it Assumptions $A(f_{\theta},\Theta), A(M_{\theta},\Theta)$, $A(\partial _\theta f_{\theta},\Theta)$,
$A(\partial _\theta M_{\theta},\Theta)$ and $B(\Theta)$ hold and there exists $r\ge 2$ such that $\theta^* \in \Theta(r)$. Moreover, with $s=\min(1,r/3)$, assume that the sequence $(\kappa_n)_{n\in \mathbb{N}}$ satisfies
\[
\sum_{k\ge 1}(\frac{1}{\kappa_k})^s\Big(\sum_{j \ge k} \alpha_j (f_\theta,\Theta)+ \alpha_j (M_\theta,\Theta)+\alpha_j (\partial _\theta f_\theta,\Theta)+\alpha_j (\partial _\theta M_\theta,\Theta)\Big)^s < \infty.
\] }

\medskip
\noindent \textbf{Assumption $\widetilde{\boldsymbol{K}}(\Theta)$}: {\it Assumptions $A(\widetilde H_{\theta},\Theta)$, $  A(\partial _\theta \widetilde  H_{\theta},\Theta)$ and $B(\Theta)$ hold and there exists $r\ge 2$ such that $\theta^* \in \Theta(r)$. Moreover, with $s=\min(1,r/4)$, assume that the sequence $(\kappa_n)_{n\in \mathbb{N}}$ satisfies
\[
\sum_{k\ge 1}(\frac{1}{\kappa_k})^s\Big(\sum_{j \ge k} \alpha_j (\widetilde H_\theta,\Theta)+ \alpha_j (\partial _\theta \widetilde H_\theta,\Theta)\Big)^s < \infty.
\] }
%
%
\begin{remark} \label{rem:1}
These conditions on $(\kappa_n)_{n\in \mathbb{N}}$ have been deduced from conditions for strong law of large numbers obtained in \cite{kounias} and are not too restrictive:
for instance, if the Lipschitzian coefficients of $f_{\theta}$, $M_{\theta}$ (the case using $\widetilde H_{\theta}$ can be treated similarly) and their derivatives
are bounded by a geometric or Riemanian decrease:
\begin{enumerate}
\item the geometric case: $\alpha_j (f_\theta,\Theta)+ \alpha_j (M_\theta,\Theta)+\alpha_j (\partial _\theta f_\theta,\Theta)+\alpha_j (\partial _\theta M_\theta,\Theta)=O(a^j)$ with $0\le  a < 1$, then any $(\kappa_n)$ such as $1/\kappa_n=o(1)$ can be chosen; for instance  $\kappa_n=\log {n}$ or $\log(\log{n})$;
this is the case for instance of ARMA, GARCH, APARCH or ARMA-GARCH processes.
\item the Riemanian case: $\alpha_j (f_\theta,\Theta)+ \alpha_j (M_\theta,\Theta)+\alpha_j (\partial _\theta f_\theta,\Theta)+\alpha_j (\partial _\theta M_\theta,\Theta)=O(j^{-\gamma})$ with $\gamma> 1$:
\begin{itemize}
\item if $r \ge 3$ then
\begin{itemize}
\item if  $\gamma >2$ then any sequence such as  $1/\kappa_n=o(1)$ can be chosen;
\item if $1 <\gamma <2$, any $(\kappa_n)$ such as $\kappa_n= O(n^{\delta})$ with $\delta>2-\gamma$ can be chosen.
\end{itemize}
\item if $1 \leq r< 3$
\begin{itemize}
\item if  $\gamma >(r+3)/r$ then any sequence such as  $1/\kappa_n=o(1)$ can be chosen;
\item if $1 <\gamma <(r+3)/r$  then any $(\kappa_n)$ such as $\kappa_n= n^{\delta}$ with $\delta>(r+3)/r-\gamma$ can be chosen.
\end{itemize}
In the last case of these two conditions on $r$, we can see the usual BIC choice, $\kappa_n=\log n$ does not fulfill the assumption in general.
\end{itemize}
\end{enumerate}
\end{remark}
\subsection{New versions of limit theorems in \cite{barW}}
\noindent These assumptions K$(\Theta)$  and $\widetilde K(\Theta)$ used in Lemmas \ref{lem1} and \ref{lem2} (see Section \ref{Proofs}) and the detailed Riemanian convergence rates of the previous remark, provide an improvement of the two main limit theorems established in \cite{barW}. More precisely, we obtain: \\
~\\
{\bf New version of Theorem 1 in \cite{barW}} \\
{\it Let $(X_1,\ldots, X_n)$ be an observed trajectory of an affine causal process $X$ belonging to $\mathcal{AC}(M_{\theta^*},f_{\theta^*})$ (or $\widetilde {\mathcal{AC}}(\widetilde H_\theta)$) where $\theta^*$ is an unknown vector of $\Theta$, a compact set included in $\Theta(r) \subset \R^d$ (or $\widetilde  \Theta(r) \subset \R^d$) with \textcolor{black}{$r\geq 2$}. Then, if   assumptions A($f_\theta,\Theta$), A($M_\theta,\Theta$) (or A($\widetilde H_\theta,\Theta$)), $D(\Theta)$, $Id(\Theta)$  hold with }
\begin{equation}\label{condP}
\left \{ \begin{array}{ll} \alpha_j (f_\theta,\Theta)+ \alpha_j (M_\theta,\Theta)=O\big (j^{-\ell}\big ) & \mbox{for some}~~\ell
>\max (1 \, , \, 3/r) \\
\mbox{or}\quad\alpha_j (\widetilde H_\theta,\Theta)=O\big (j^{-\widetilde \ell}\big ) & \mbox{for some}~~\widetilde \ell
>\max (1 \, , \, 4/r)
\end{array} \right .,
\end{equation}
{\it then the QMLE $\widehat \theta(m^*)$ satisfies
$
\widehat \theta(m^*) \limiteasn \theta^*.
$ }
~\\
\begin{proof}
We use the same proof as in \cite{barW} except for establishing $\frac{1}{n} \,\big \|\widehat{L}_n(\theta)- L_n(\theta)\big \|_{\Theta}\limiteasn 0$.
Indeed, we can apply Lemma \ref{lem1} with $\kappa_n=n$. Hence, this is checked under assumption $\boldsymbol{K}(\Theta)$ under Riemanian condition of Remark \ref{rem:1} if $r\geq 3$ when $\gamma=\ell>1$ and \textcolor{black}{if $2 \leq r \leq 3$, when $\gamma=\ell>3/r$}, implying the first new conditions of the Theorem. \\
Under assumption $\widetilde{\boldsymbol{K}}(\Theta)$, an adaptation of Remark \ref{rem:1} implies that for $r\geq 4$ we should have $\gamma=\widetilde \ell>1$ and if \textcolor{black}{$2 \leq r \leq 4$}, when $\gamma=\widetilde \ell>4/r$.
\end{proof}
\noindent Therefore, in all the previous cases and when $r=4$, we obtain a limiting decrease rate $O\big (j^{-\gamma}\big )$ with $\gamma>1$ instead of  $\gamma>3/2$ in \cite{barW}. This can also be used to improve Theorem 2 in \cite{barW}:\\
~\\
{\bf New version of Theorem 2 in \cite{barW}} \\
{\it If $r\geq 4$ and under the assumptions of the previous new version of Theorem 1 in \cite{barW}, and {\bf Var($\Theta$)}, and if   assumptions A($\partial _\theta f_\theta,\Theta$), A($\partial _\theta M_\theta,\Theta$), A($\partial^2 _{\theta^2} f_\theta,\Theta$) and A($\partial^2 _{\theta^2} M_\theta,\Theta$) (or A($\partial _\theta\widetilde H_\theta,\Theta$) and A($\partial^2 _{\theta^2} \widetilde H_\theta,\Theta$) )
hold with }
\begin{equation}\label{condth2}
\left \{ \begin{array}{c}  \alpha_j(\partial _\theta f_\theta,\Theta)+\alpha_j(\partial _\theta  M_\theta,\Theta)=O\big (j^{-\ell'}\big
) \\ \quad\mbox{or}\quad \alpha_j(\partial _\theta \widetilde H_\theta,\Theta)=O\big (j^{-\ell'}\big
)
\end{array}\right . ~~\mbox{for some}~~\ell' >1,
\end{equation}
{\it  then the QMLE $\widehat \theta_n(m^*)$ satisfies}
\begin{eqnarray}\label{tlcqmle}
\sqrt{n} \, \Big (\big (\widehat\theta(m^*)\big )_i-(\theta^*)_i\Big )_{i \in m^*}\limiteloin {\cal
N}_{|m^*|}\big (0 \ , \ F(\theta^*,m^*)^{-1} G(\theta^*,m^*)
F(\theta^*,m^*)^{-1}\big ),
\end{eqnarray}
{\it with $\displaystyle \big (F(\theta^*,m^*)\big )_{i,j}=\mathbb{E}\Big[\frac{\partial^2 q_0(\theta^*)}{\partial \theta_i \partial \theta_j}\Big]$   and $\displaystyle (G(\theta^*,m^*))_{i,j}=\mathbb{E}\Big[\frac{\partial q_0(\theta^*)}{\partial \theta_i} \frac{\partial q_0(\theta^*)}{\partial \theta_j} \Big]$ for $i,\, j \in m^*$.}\
\subsection{Asymptotic model selection}
\noindent Using the above assumptions, we can establish the limit theorem below, which provides sufficient conditions for the consistency of the model selection procedure.
\begin{theo}\label{theo:1}
Let $(X_1,\ldots, X_n)$ be an observed trajectory of an affine causal process $X$ belonging to $\mathcal{AC}(M_{\theta^*},f_{\theta^*})$
(or $\widetilde {\mathcal{AC}}(\widetilde H_\theta)$) where $\theta^*$ is an unknown vector of $\Theta$ a compact set included in
$\Theta(r) \subset \R^d$ (or $\widetilde  \Theta(r) \subset \R^d$) with $r\geq 4$. If assumptions $D(\Theta)$, $Id(\Theta)$, $K(\Theta)$ (or $\widetilde K(\Theta)$),
A($\partial^2 _{\theta^2} f_\theta,\Theta$) and  A($\partial^2 _{\theta^2} M_\theta,\Theta$) (or  A($\partial^2 _{\theta^2} \widetilde H_\theta,\Theta$)) also hold, then
\begin{equation}
\P(\widehat{m}=m^*)\limiten 1 \quad \mbox{and}\quad
\widehat\theta(\widehat{m}) \limiteproban \theta^*.
\label{eq:F}
\end{equation}
\end{theo}
\noindent The following theorem shows the  asymptotic normality of the QMLE of the chosen model.
\begin{theo}\label{theo:2}
Under the assumptions of Theorem \ref{theo:1} and if $\theta^* \in \interior{\Theta}$ and {\bf Var($\Theta$)} hold, then
\begin{eqnarray}
\sqrt{n} \, \Big (\big (\widehat\theta(\widehat{m})\big )_i-(\theta^*)_i\Big )_{i \in m^*}\limiteloin {\cal
N}_{|m^*|}\big (0 \ , \ F(\theta^*,m^*)^{-1} G(\theta^*,m^*)
F(\theta^*,m^*)^{-1}\big ),
\label{eq:con2}
\end{eqnarray}
where $F$  and $G$ are defined in \eqref{tlcqmle}.
\end{theo}

\begin{remark}\label{rem:2}
 In Remark \ref{rem:1}, we detailed some situations where the assumption $K(\Theta)$ (or $\widetilde K(\Theta)$) holds, which leads to the results
 of Theorem \ref{theo:1} and \ref{theo:2}.
 In particular, the $\log n$ penalty usually linked to BIC is consistent in the case of a geometric decrease of the Lipschitz coefficients of the functions $f_{\theta}$ and $M_{\theta}$ (and their first order
 derivative).
 In the case of a Riemanian rate, the consistency of BIC is not ensured; see also the next section.
\end{remark}

\section{Examples}\label{Examples}
In this section, some examples of time series satisfying the conditions of previous results are considered. These examples include $AR(\infty)$, $ARCH(\infty)$, $APARCH(\delta,p,q)$ and ARMA$(p,q)$-GARCH$(p',q')$.

\subsection{AR$(\infty)$ models}
For $(\psi_k(\theta))_{k\in \N}$ a sequence of real numbers depending on $\theta\in \R^d$, let us consider an $AR(\infty)$ process defined by:
\begin{equation}
X_t=\sum_{k\ge 1}\psi_k(\theta^*)X_{t-k}+ \sigma \, \xi_t\quad\mbox{for any $t \in \Z$},
\label{eq:ex1}
\end{equation}
where $(\xi_t)_t$ admits $4$-order moments, and $ \theta^* \in \Theta \subset \Theta(4)$, the set of $\theta \in \R^d$ such that $\sum_{k\ge 1}\|\psi_k(\theta)\|_{\Theta}<1$ and $\sigma>0$.
This process corresponds to \eqref{eq:serie} with $f_{\theta}\big ((x_i)_{i\geq 1}\big )=\sum_{k\ge 1}\psi_k(\theta)x_{k}$ and $M_{\theta} \equiv \sigma$ for any $\theta
 \in \Theta$. The Lipschitz coefficients of $f_{\theta}$ are $\alpha_k(f_{\theta})=\|\psi_k(\theta)\|_{\Theta}$. Moreover, Assumption $D(\Theta)$  holds with $\underline{h}=\sigma^2>0$.

 \medskip

 \noindent Let us consider $\mathcal{M}$ a finite family of models. Of course, the main example of such family of models is given by the one of ARMA$(p,q)$ processes with $0\leq p \leq p_{\max}$ and $0\leq q \leq q_{\max}$, providing $(p_{\max}+1)(q_{\max}+1)$ models and $\theta \in \R^{p_{\max}+q_{\max}+1}$.

 \medskip

 \noindent Besides, assume that $Id(\Theta)$, 
  Var($\Theta$) hold and that the sequence $(\psi_k)$ is twice differentiable (with respect to $\theta$) on $\Theta$, with $\sum_k\|\partial^2_{\theta} \psi_k(\theta)\|_{\Theta}<\infty$ and $\|\psi_k(\theta)\|_{\Theta}+\|\partial_\theta \psi_k(\theta)\|_{\Theta}=O(k^{-\gamma})$ with $\gamma>1$. From Remark \ref{rem:1},
\begin{itemize}
\item if $\gamma>2$, the condition $\kappa_n\limiten \infty$ (for instance, the BIC penalization with $\kappa_n=\log(n)$, or $\kappa_n=\sqrt{n}$) ensures the consistency of $\widehat{m}$ and the Theorem \eqref{theo:2} holds if in addition $\theta^* \in \interior{\Theta}$;
\item if $1< \gamma <2$, $\kappa_n= O(n^{\delta})$ with $\delta>2-\gamma$ has to be chosen (and we cannot insure the consistency of $\widehat{m}$ in case of classical BIC penalization).
\end{itemize}
Finally, in the particular case of the family of ARMA processes, the stationarity condition implies that any $\kappa_n\limiten \infty$ can be chosen (BIC penalization with $\kappa_n=\log(n)$, or $\kappa_n=\sqrt{n}$), since the decreases of $\psi_k$ and its derivative are exponential.

\subsection{ARCH$(\infty)$ models}
For $(\psi_k(\theta))_{k\in \N}$ a sequence of nonnegative real numbers depending on $\theta\in \R^d$, with $\psi_0>0$, let us consider an ARCH$(\infty)$ process
defined by :
\begin{equation}
X_t=\Big(\psi_0(\theta^*)+  \sum_{k=1}^{\infty}\psi_k(\theta^*)X_{t-k}^2\Big)^{1/2}\xi_t \quad\mbox{for any $t \in \Z$},
\label{eq:ex2}
\end{equation}
where $\mathbb{E} \big [\xi^4_0\big ]<\infty$,  and $ \theta^* \in  \Theta \subset \widetilde \Theta(4)$,
the set of $\theta \in \R^d$ such that $\sum_{k\ge 1}\|\psi_k(\theta)\|_{\Theta}<1$. 
This process corresponds to \eqref{eq:serie} with $f_{\theta}\big ((x_i)_{i\geq 1}\big )\equiv 0$ and $H_{\theta} \big ((x_i)_{i\geq 1}\big )= \psi_0(\theta)+  \sum_{k=1}^{\infty}\psi_k(\theta)x_k^2$, {\it i.e.} $\widetilde H_{\theta} \big ((y_i)_{i\geq 1}\big )= \psi_0(\theta)+  \sum_{k=1}^{\infty}\psi_k(\theta)y_k$, for any $\theta\in \Theta$.
The Lipschitz coefficients of $\widetilde H_{\theta} $ are $\alpha_k(\widetilde H_{\theta} )=\|\psi_k(\theta)\|_{\Theta}$. Moreover, Assumption $D(\Theta)$  holds if $\underline{h}=\inf_{\theta\in \Theta}\psi_0(\theta)>0$.

 \medskip

\noindent Let us consider $\mathcal{M}$ a finite family of models. The main example of such family of models is given by the GARCH$(p,q)$ processes with $0\leq p \leq p_{\max}$ and $0\leq q \leq q_{\max}$, providing $(p_{\max}+1)(q_{\max}+1)$ models and $\theta \in \R^{p_{\max}+q_{\max}+1}$.

\medskip

\noindent Moreover, assume that $Id(\Theta)$, Var($\Theta$) hold and that the sequence $(\psi_k)$ is twice differentiable (with respect to $\theta$) on $\Theta$, with $\sum_k\|\partial^2_{\theta} \psi_k(\theta)\|_{\Theta}<\infty$ and $\|\psi_k(\theta)\|_{\Theta}+\|\partial_\theta \psi_k(\theta)\|_{\Theta}=O(k^{-\gamma})$ with $\gamma>1$. From Remark \ref{rem:1},
\begin{itemize}
\item if $\gamma>2$, the condition $\kappa_n\limiten \infty$ (for instance, the BIC penalization with $\kappa_n=\log(n)$, or $\kappa_n=\sqrt{n}$) ensures the consistency of $\widehat{m}$ and the Theorem \eqref{theo:2} holds if in addition, $\theta^* \in \interior{\Theta}$;
\item if $1< \gamma <2$, $\kappa_n= O(n^{\delta})$ with $\delta>2-\gamma$ has to be chosen (and we cannot insure the consistency of $\widehat{m}$ in the case of the classical BIC penalization).
\end{itemize}
Finally, in the particular case of the family of GARCH processes, the stationarity condition  implies that any $\kappa_n\limiten \infty$ can be chosen (BIC penalization with $\kappa_n=\log(n)$, or $\kappa_n=\sqrt{n}$), since the decreases of $\psi_k$ and its derivative are exponential.
\subsection{APARCH$(\delta, p, q)$ models}
For $\delta\geq 1$ and from \cite{ding}, $(X_t)_{t\in \Z}$ is an APARCH$(\delta, p, q)$ process with $ p , q \geq 0$ if:
\begin{equation}
\begin{cases}
           X_t=\sigma_t\, \xi_t \\
          (\sigma_t)^{\delta}=\omega+\sum_{i=1}^{p} \alpha_i(|X_{t-i}|-\gamma_i X_{t-i})^{\delta}+ \sum_{j=1}^{q} \beta_j (\sigma_{t-j})^{\delta}\quad\mbox{for any $t \in \Z$},
            \end{cases}
\label{eq:ex3}
\end{equation}
where $\omega>0$, $ -1<\gamma_i <1$, $\alpha_i\ge 0$, $\beta_j\ge 0$ for $1\leq i \leq p$ and $1\leq j \leq q$, $\alpha_{p}>0$, $\beta_{q}>0$ and $\sum_{j=1}^q \beta_j<1$.
From \cite{barY}, with $\theta=(\omega,\alpha_1,\ldots,\alpha_{p},\gamma_1,\ldots,\gamma_{p},\beta_1,\ldots,\beta_{p})'$, the conditional variance $\sigma_t$ can be rewritten
 as follows
\[
\sigma_t^{\delta}=b_0(\theta)+ \sum_{k\ge 1}\Big(b_k^+(\theta)(\max(X_{t-k},0))^{\delta}-b_k^-(\theta)(\min(X_{t-k},0))^{\delta} \Big);
\]
with $f_\theta \equiv 0$ and $M_\theta^t=\sigma_t$, we deduce that $\alpha_k(M_\theta,\Theta)=\max(\|b_k^+(\theta)\|^{1/\delta}_{\Theta},\|b_k^-(\theta)\|^{1/\delta}_{\Theta})$,
and from the assumption $\sum_{j=1}^q \beta_j<1$, the Lipschitz coefficients $\alpha_k(M_\theta,\Theta)$ decrease exponentially fast. Then, the stationarity set for $r\geq 1$ is
\[
\Theta(r)=\Big \{\theta\in\R^{2p+q+1}~\Big/~  \|\xi_0\|_r \, \sum_{j=1}^\infty \max \big (|b_j^+(\theta)|^{1/\delta},|b_j^-(\theta)|^{1/\delta}\big )< 1\Big \}.
\]
 Now, assume that $(X_t)_{t\in \Z}$ is an APARCH$(\delta, p^*, q^*)$ where $0\leq p^* \leq p_{\max}$ and $0\leq q^* \leq q_{\max}$ are unknown orders as well as the other parameters:  $\omega^*>0$, $ -1<\gamma^*_i <1$, $\alpha^*_i\ge 0$, $\beta^*_j\ge 0$ for $1\leq i \leq p_{\max}$ and $1\leq j \leq q_{\max}$, $\alpha_{p^*}>0$, $\beta_{q^*}>0$.

\medskip

\noindent Let ${\cal M}$ be the family of APARCH$(\delta, p, q)$ processes, with $0\leq p \leq p_{\max}$ and $0\leq q \leq q_{\max}$. As a consequence, we consider here $d=2p_{\max}+q_{\max}+1$, and
\[
\theta^*={}^t \big (\omega^*,\alpha^*_1,\ldots,\alpha^*_{p^*},0,\ldots,0,\gamma^*_1,\ldots,\gamma^*_{p^*},0,\ldots,0,\beta^*_1,\ldots,\beta^*_{q^*},0,\ldots,0\big ) \in \R^d.
\]
With all the previous conditions, assumptions D$(\Theta)$, Id$(\Theta)$, Var($\Theta$) are satisfied. Moreover, since the Lipschitz coefficients decrease exponentially fast,
K$(\Theta)$ is  satisfied when $\kappa_n \to \infty$. Therefore, the consistency Theorem \eqref{theo:1} and the Theorem \eqref{theo:2} of the estimator of the chosen model are satisfied when $r=4$ and $\kappa_n \to \infty$ (for instance with the typical BIC penalty $\kappa_n=\log n$).

\subsection{ARMA$(p,q)$-GARCH$(p',q')$ models}
From \cite{ding} and \cite{ling}, we define  $(X_t)_{t\in \Z}$ as an (invertible) ARMA$(p,q)$-GARCH$(p',q')$ process with $ p, q, p', q' \geq 0$ if:
\[
\begin{cases}
          X_t=\sum_{i=1}^p a_i \, X_{t-i} +\varepsilon_t-\sum_{i=1}^q b_i\, \varepsilon_{t-i}\\
          \varepsilon_t=\sigma_t \, \xi_t, \; \mbox{with}\; \sigma_t^2=c_0+ \sum_{i=1}^{p'}c_i \,\varepsilon_{t-i}^2+\sum_{i=1}^{q'}d_i \, \sigma_{t-i}^2
            \end{cases} \quad \mbox{for all $t\in \Z$},
\]
where
\begin{itemize}
\item $c_0> 0$, $c_{p'}>0$, $c_i \ge 0$ for $i=1,\cdots,p'-1$ and $d_{q'}>0$, $d_i \ge 0$ for $i=1,\cdots,q'-1$;
\item $ P(x)=1-\sum_{i=1}^p a_ix^{i}$ and $ Q(x)=1-\sum_{i=1}^q b_ix^{i}$ are coprime polynomials.
\end{itemize}
Here we will consider the case of a stationary invertible ARMA$(p,q)$-GARCH$(p',q')$ process such as $\|X_0\|_4<\infty$ and therefore we will consider:
\begin{multline*}
\Theta_{p,q,p',q'}=\Big \{ (a_1,\ldots,d_{q'})\in \R^{p+q+p'+1+q'},~\sum_{j=1}^{q'} d_j+\| \xi_0\|_4 \, \sum_{j=1}^{p'} c_j<1  \\
\mbox{and}~\big (1-\sum_{j=1}^{p} a_j z^j\big ) \, \big (1-\sum_{j=1}^{q} b_j z^j\big ) \neq 0~\mbox{for all $|z|\leq 1$} \Big \}.
\end{multline*}
Therefore, if $(a_1,\ldots,d_{q'}) \in \Theta_{p,q,p',q'}$,  $(\varepsilon_t)_t$ is a stationary GARCH$(p',q')$ process  and $(X_t)_t$ is a stationary weak invertible ARMA$(p,q)$ process. \\
Moreover, following Lemma 2.1. of \cite{barY}, we know that a stationary ARMA$(p,q)$-GARCH$(p',q')$ process is a stationary affine causal process with functions $f_{\theta}$ and $M_{\theta}$ satisfying the Assumption A$(f_{\theta},\Theta)$ and A$(M_{\theta},\Theta)$ with Lipschitzian coefficients decreasing exponentially fast, as well as their derivatives. Finally,  if $\Theta$ is a bounded subset of  $ \Theta_{p,q,p',q'}$, then assumptions D$(\Theta)$, Id$(\Theta)$ and Var($\Theta$) are automatically satisfied.

\medskip

\noindent Assume now that $(X_t)_{t\in \Z}$ is an ARMA$(p^*,q^*)$-GARCH$(p^{'*},q^{'*})$ process where $0\leq p^* \leq p_{\max}$, $0\leq q^* \leq q_{\max}$, $0\leq p^{'*} \leq p'_{\max}$ and $0\leq q^{'*} \leq q'_{\max}$ are unknown orders with also unknown parameters: $c^*_0, \ldots,c^*_{p^{'*}},d^*_1,\ldots,d^*_{q^{'*}},a^*_1,\ldots, a^*_{p^*}, b^*_1,\ldots,b_{q^*}$ .

\medskip

\noindent Let ${\cal M}$ be the family of ARMA$(p,q)$-GARCH$(p^{'},q^{'})$ processes, with $0\leq p \leq p_{\max}$, $0\leq q \leq q_{\max}$, $0\leq p' \leq p'_{\max}$ and $0\leq q' \leq q'_{\max}$. Hence, we consider here $d=p_{\max}+q_{\max}+p'_{\max}+q'_{\max}+1$, and
\[
\theta^*=\big (c^*_0, \ldots,c^*_{p^{'*}},0,\ldots,0,d^*_1,\ldots,d^*_{q^{'*}},0,\ldots,0,a^*_1,\ldots, a^*_{p^*},0,\ldots,0, b^*_1,\ldots,b_{q^*},0,\ldots,0\big ) \in \R^d.
\]
With $\Theta$ a bounded subset of  $ \Theta_{p_{\max},q_{\max},p'_{\max},q'_{\max}}$, all the previous assumptions D$(\Theta)$, Id$(\Theta)$, Var($\Theta$) are satisfied and K$(\Theta)$ is also satisfied as soon as $\kappa_n \to \infty$. As a consequence, in this framework the consistency Theorem \eqref{theo:1} and the Theorem \eqref{theo:2} of the estimator of the chosen model are satisfied when $r=4$ and $\kappa_n \to \infty$ (for instance with the typical BIC penalty $\kappa_n=\log n$).
\section{Portmanteau test}\label{Test}
From the above section, we are now able to asymptotically pick up a best model in a family of models. We can also obtain asymptotic confident regions of the estimated
parameter of the chosen model. However, it is also important to check whether the chosen model is appropriate.
 This section attempts to answer this question by constructing a portmanteau test as a diagnostic tool based on the squares of the residuals sequence of the chosen model.

\noindent This test has been widely considered in the time series literature, with procedures based on the squared residual correlogram (see for instance \cite{li2}, \cite{Ling1997} )
and the absolute  residual (or usual residuals) correlogram (see for instance \cite{li}, \cite{Francq2008}, \cite{li2008}), among others.

\noindent Since our goal is to provide an efficient test for the entire affine class that contains weak white noise processes, we consider in this setting the autocorrelation of the squared residuals and then we will follow the same scheme of procedure used in (\cite{li2}, \cite{Ling1997}) while relying on some of their results.

\medskip
\noindent For $m\in {\cal M}$, for $K$ a positive integer, denote the vector of adjusted correlogram of squares residuals by:
\begin{equation*}
\widehat{\rho}(m):=\big ( \widehat\rho_1(m), \ldots,\widehat\rho_{K}(m) \big )',
\end{equation*}
where for $k=1,\ldots,K$, $\displaystyle \widehat \rho_k(m):=\displaystyle \frac {\widehat \gamma_k(m)}{\widehat \gamma_0(m)}$ with
\begin{equation*}
 \widehat \gamma_k(m):=\displaystyle \frac 1 n \, \sum_{t=k+1}^{n} \big (\widehat e_t^2(m) -1 \big )\big (\widehat e_{t-k}^2(m) -1 \big )\quad\mbox{and}\quad\widehat e_t(m):=\displaystyle \big ( \widehat M_{\widehat \theta(m)}^{t}\big )^{-1} \big (X_t-\widehat f_{\widehat \theta(m)}^t\big).
\end{equation*}
Finally, the following theorem provides central limit theorems for $\widehat{\rho}(m^*)$ and $\widehat{\rho}(\widehat m)$ as well as for a portmanteau test statistic.

\begin{theo}
Under the assumptions of Theorem \ref{theo:2}, if
\begin{itemize}
\item $\E [\xi_0^3]=0$;
\item $\displaystyle
\sum_{t=1}^\infty t ^{-1/4} \Big (\sum_{j \ge t} \alpha_j (f_\theta,\Theta)+\alpha_j (M_\theta,\Theta) \Big )^{1/2}<\infty$ or $\displaystyle
\sum_{t=1}^\infty t ^{-1/4} \Big (\sum_{j \ge t} \alpha_j (\widetilde H_\theta,\Theta) \Big )^{1/2}<\infty$;
\end{itemize}
then,
\begin{enumerate}
\item With $V(\theta^*,m^*)$ defined in \eqref{VV}, it holds that
\begin{equation}
\sqrt n \, \widehat{\rho}(m^*) \limiteloin {\cal N}_{K} \big ( 0 \, , \, V(\theta^*,m^*) \big ).
\label{eq:por}
\end{equation}
\item With $\widehat Q_K(m^*):=n \, \widehat{\rho}(m^*)'\big ( V(\widehat \theta(m^*),m^*)\big )^{-1} \widehat{\rho}(m^*) $, we have
\begin{equation}
\widehat Q_K(m^*) \limiteloin \chi^2(K).
\label{eq:chi}
\end{equation}
\item The previous points 1. and 2. also hold when $m^*$ is replaced by $\widehat m$.
\end{enumerate}
\label{theo:3}
\end{theo}
\noindent Using the Theorem \ref{theo:3}, we can asymptotically test:\\
~\\
\qquad\qquad $\displaystyle \left \{ \begin{array}{l} H_0:~\mbox{$\exists m^*\in {\cal M}$, such as $(X_1,\ldots,X_n)$ is a trajectory of $X \in {\cal AC}(M_{\theta},f_{\theta^*})$ with $\theta^*\in \Theta(m^*)$}\\
~\\
H_1:~\mbox{$\nexists m^*\in {\cal M}$, such as $(X_1,\ldots,X_n)$ is a trajectory of $X \in {\cal AC}(M_{\theta},f_{\theta^*})$ with $\theta^*\in \Theta(m^*)$}\end{array} . \right . $ \\
~~\\
~~\\
\noindent Therefore, $\widehat Q_K(\widehat m) $ can be used as a portmanteau test statistic to decide between $H_0$ and $H_1$ and diagnose the goodness-of-fit of the selected model.\\
\begin{remark}
\begin{enumerate}
\item Like in \cite{li2}, it is important to point out that for $ARCH(p)$ model, since
$f_{\theta}^t=0$, we have $\E\Big [ \big(\xi_{0}^2 -1\big)\, \partial_\theta \log \big (M_{\theta^*}^k\big )\Big ]=0 $ for all $k > p$.
Hence, for these models, the matrix $V(\theta^*,m^*)-I_K $ will have approximately zero entries from the $(p+1)^{th}$ row onwards and then the standard errors of
$\widehat{\rho}(m^*)_i $ are in this case equal to $ 1/\sqrt{n}$ for $ i= p+1, \ldots, K$.
The statistic $\widehat{Q}_K(m^*)$ yields to  $\widehat{Q}(p,K):=n\sum_{i=p+1}^K [\widehat{\rho}(m^*)_i ]^2 $ which will be asymptotically $\chi^2$ distributed with $K-p$ degrees of freedom.

\item In practice the constant $\mu_4$ and the rows of the matrix $V(\widehat \theta(m^*),m^*)$ involved in the previous theorem are estimated by the correspondent sample average; they are respectively  $\widehat{\mu}_4=\frac{1}{n} \sum_{t=1}^{n}(\widehat{e}_t(\widehat m))^4$ and $\big(\widehat V(\widehat \theta((\widehat m)),(\widehat m))\big)_{k,.}=\frac{1}{n} \sum_{t=k+1}^{n}[(\widehat{e}_t(\widehat m))^2-1][\partial_\theta \log \big (M_{\theta}^k)]_{(\theta=\widehat{\theta}(\widehat m))}.$
\end{enumerate}
\end{remark}

\section{Numerical Results}\label{Numeric}

This section features some simulation experiments that are performed to assess the usefulness of the asymptotic results obtained in Section \ref{sec:3}. The various configurations studied are presented below and we compare the performance of penalties $\log n$ and $\sqrt{n}$. The process used to generate the trajectory is indicated each time.

\medskip

\noindent Each model is generated independently  1000 times over a trajectory of length $n$. Different sample sizes are considered to identify possible discrepancies between asymptotically expected properties and those obtained at finite distance. We will consider $n$ belongs to $\{100,500,1000,2000\}$.
Throughout this section, $(\xi_t)$ represents a Gaussian white noise with variance unity.

\subsection{Classical configurations}
We first simulate some classical model illustrated as follows  and the results are displayed in the Table \ref{tab:1}.

\begin{enumerate}
\item Model 1, AR$(2)$ process: $X_t=0.4X_{t-1}+0.4X_{t-2}+\xi_t$.
\item Model 2, ARMA$(1,1)$ process: $X_t=0.3X_{t-1}+\xi_t+0.5\xi_{t-1}$.
\item Model 3, ARCH$(2)$ process: $X_t=\xi_t \sqrt{0.2+0.4X_{t-1}^2+0.2X_{t-2}^2}$.
\end{enumerate}
We considered as competitive models all the models in the family ${\cal M}$ defined by:
$$
{\cal M}=\big \{\mbox{ARMA$(p,q)$ or GARCH$(p',q')$ processes with $0\le p,q,p' \le 5$, $1\le q'\le 5$} \big \}.
$$
As a consequence, there are $66$ candidate models. \\
~\\
The Table \ref{tab:1} shows for each penalty ($\log n$ and $\sqrt{n}$) the percentage of times the associated criterion  selects respectively a wrong model, the true model and an overfitted model (here a model which contains the true model).

\begin{center}
\captionof{table}{Percentage of selected order  based on 1000 replications depending on sample's length for Model 1, 2 and 3 respectively.} \label{tab:1}
\begin{tabular}{|l|c|lr|lr|lr|lr|}
\toprule
\multicolumn{1}{|c|}{}& \multicolumn{1}{|c|}{Sample length $n$ }& \multicolumn{2}{c}{$100$ } & \multicolumn{2}{c}{$500$ }& \multicolumn{2}{c}{$1000$ } & \multicolumn{2}{c|}{$2000$ }\\
& Penalty &  $\log n$ & $\sqrt{n}$  &  $\log n$ & $\sqrt{n} $&  $\log n$ & $\sqrt{n} $&  $\log n$ & $\sqrt{n} $ \\
\midrule
&Wrong &21     &32.3   &3    &0.9   &0.9   &0   &0.2  &0     \\
Model 1&True &74.6    &67.5   &95.8 &99.1 &98.2    &100    &99  & 100\\
& Overfitted &4.4     &0.2    &1.2  &0   &0.9 &0  &0.8  &0 \\
\hline
\midrule
& Wrong &81.8 &97.5 &30.1 &67.4 &19.9 &33.2 &10.2  &10.5     \\
Model 2 & True &16.1 &2.5 &69.1 &32.6 &79.5 &66.8 &89.4  &89.5 \\
& Overfitted &2.1 &0 &0.8 &0 &0.6 &0  &0.4  &0 \\
\hline
\midrule
& Wrong &78.9 &92.9 &25.7 &70.5 &11.6. &39.2 &5.4 & 11.4    \\
Model 3 &True &20.4 &7.0 &73.2 &29.5 &88.1 &60.8 &94.3 &88.6\\
& Overfitted &0.1 &0.1 &1.1 &0 &0.3 &0  & 0.3 & 0\\
\bottomrule
\end{tabular}
\end{center}

\vspace{0.5cm}
From these results, it is clear that the consistency of our model selection procedure is numerically convincing, which is in accordance with Theorem \ref{theo:1}, where both the criteria are consistent for Model 1, 2 and 3. Note also that the typical BIC $\log n$ penalty is the most interesting for retrieving the true model than the  $\sqrt{n}$-penalized likelihood  for a small sample size.
But the larger the sample size, the more accurate the $\sqrt n$ penalty case. \\
~\\
For each of the three models, we also applied the portmanteau test statistic $\widehat Q_K(\widehat m) $, using the $\sqrt n$ penalty.
Table \ref{tab:2} shows the empirical size and empirical power of this test. We call by empirical size, the percentage of  falsely rejecting the null hypothesis $H_0$.
On the other hand, the empirical power represents the percentage of rejection of $H_0$ when we arbitrary chose a false model, which is a AR$(3)$ process
 $X_t=0.2X_{t-1}+0.2X_{t-2}+0.4X_{t-1}+\xi_t$ for Model 1 and 2, and a ARCH$(3)$ process $X_t=\xi_t \, \sqrt{0.4+0.2X_{t-1}^2+0.2X_{t-2}^2+0.2X_{t-3}^2}$ for Model 3. \\
~\\
It is important to note that choosing the maximum number of lags $K$ is sometimes tricky. To our knowledge, there is no real theoretical study to justify the choice of one
value or another. However, some Monte Carlo simulations have suggested some ways to make a good choice . For instance \cite{li2} suggested that the autocorrelations
 $\widehat{\rho}_k(\widehat m)$ with $1\le k\le K$ have a better asymptotic behaviour for small values of $k$. Therefore,  the finite sample performance of the
  size and power of the test may also vary with the choice of $K$ and could be better for small values of $K$. On the other hand, \cite{tse} suggested that $K=p+q+1$ may be
  an appropriate choice for the GARCH$(p,q)$ family.\\
Thus, in our tests, we consider $K=3 \; \mbox{and} \; K=6 $  so that the rejection is based on the upper 5th percentile of the $\chi^2 (3)$ distribution on the one hand and $\chi^2 (6)$ on the other hand.

\begin{center}
\captionof{table}{The empirical size and empirical power of the portmanteau test statistic $\widehat Q_K(\widehat m) $ based on 1000 independent replications (in $\% $) with $K=3$ and $K=6$.} \label{tab:2}
\begin{tabular}{|lr|lr|lr|lr|lr|}
\toprule
\multicolumn{2}{|c|}{Sample length}& \multicolumn{2}{c}{$100$ } & \multicolumn{2}{c}{$500$ }& \multicolumn{2}{c}{$1000$ } & \multicolumn{2}{c|}{$2000$ }\\
&  &  size & power&  size & power&  size & power&  size & power \\
\midrule
\multirow{3}{*}{$K=3$} & Model 1 &3.5 &13.6 &3.8 &48.1 &3.5 &82.7 &3.2  &97.7     \\
 & Model 2 &4.0 &6.7 &5 &21.7 &4.8 &38.6 &4.4  &64.2 \\
 & Model 3 &4.3 &52.7 &4.2 &98.6 &3.2 &99.6  &3.6 &99.9\\
 \hline
 \midrule
\multirow{3}{*}{$K=6$} & Model 1 &3.5 &9.4 &4.8 &43.1 &5.3 &74.6 &4.5  &97.6     \\
& Model 2 &2.1 &6.3 &4.9 &18 &4.5 &32.2 &6.4  &61.3 \\
& Model 3 &3 &18.3 &3.1 &91.5 &3.4 &99.6  &6.8 &99.7\\
\bottomrule
\end{tabular}
\end{center}
Once again, the results of Table \ref{tab:2} numerically confirms the asymptotic results of Theorem \ref{theo:3}. Remark that the test is more powerful by using values of $K$ not too large as mentioned above especially for small samples.

\subsection{Subset model selection}
Now, we exhibit the performance of the criteria on a particular case of dimension selection. The process generated data is considered as follows:

\[ \textnormal{Model 4}: X_t=0.4X_{t-3}+0.4X_{t-4}+\xi_t.\]
Here, we will consider the case of a nonhierarchical but exhaustive family ${\cal M}$ of AR$(4)$ models , {\it i.e.}
\begin{eqnarray*}
{\cal M}&=&{\cal P}(\{1,2,3,4\}) \\
&\Longrightarrow  & X_t=\theta_1 X_{t-1}+\theta_2 X_{t-2}+\theta_3 X_{t-3}+\theta_4 X_{t-4} +\xi_t~\mbox{and}~\theta=(\theta_1,\theta_2,\theta_3,\theta_4)'\in \Theta(m).
\end{eqnarray*}
As a consequence, $16=2^4$ candidate models are considered and Table \ref{tab:3} presents the results of the selection procedure.

\begin{center}
\captionof{table}{Percentage of selected model based on 1000 replications depending on sample's length for Model 4} \label{tab:3}
\begin{tabular}{|c|lr|lr|lr|lr|}
\toprule
\multicolumn{1}{|c|}{Sample length}& \multicolumn{2}{c}{$100$ } & \multicolumn{2}{c}{$500$ }& \multicolumn{2}{c}{$1000$ } & \multicolumn{2}{c|}{$2000$ }\\
 & $\log n$ & $\sqrt{n}$&  $\log n$ & $\sqrt{n}$&  $\log n$ & $\sqrt{n}$& $\log n$ & $\sqrt{n}$ \\
\midrule
true model &85.9 &68 &97.5 &100 &96.8 &100 & 98.9 & 100     \\
overfitted &7.9 &2 &2.5 &0 &3.2 &0 &1.1 & 0\\
false model &6.2 &30 &0 &0 &0 &0 & 0 & 0\\
\bottomrule
\end{tabular}
\end{center}
We deduce  that the consistency of our model selection procedure is also numerically convincing in this case of exhaustive model selection, which is in accordance with Theorem \ref{theo:1}

\subsection{Slow decrease of the Lipschitz coefficients}
In this subsection, we consider an $AR(2)-ARCH(\infty)$ with a slow decrease of its Lipschitz coefficients in order to numerically show that the penalty $\log{n}$ is not consistent in all cases. The considered data generating process is featured as follows:
\[ \textnormal{Model 5}: X_t=-0.45\,X_{t-1}+0.4\,X_{t-2}+\xi_t \;\textnormal{with}\; \xi_t=\varepsilon_t \, \sqrt{0.5+0.1 \, \sum_{i\ge 1} \xi_{t-i}^2}/ {i^3}, \]
 where $\varepsilon_t$ is an i.i.d random sequence with mean $0$ and variance $1$. The sequence $(\alpha_i)_{i\ge 1}$ verifies $\alpha_i= O(i^{-3})$ so that the sequence of
 Lipschitz coefficients of $M_{\theta}^{\xi}$ is given by  $\alpha_i (M_{\theta}^{\xi})=O(i^{-1.5})$ and then the decrease rate of the sequence
  $\big(\alpha_i (M_{\theta}^{X})\big)$ is equal to $O(i^{-1.5})$. From Remark \ref{rem:1}, all penalties such as $n^{\delta}$ with $\delta > 2-1.5=0.5$ will lead to a
  consistent model selection criterion and this is not the case for the typical BIC $\log n$ penalty. We have considered $\delta=2/3$ as in the Bridge Criteria (BC) recently proposed in \cite{ding2018bridging}. Here the family of model ${\cal M}$ is defined by
    \[ {\cal M}=\big \{\mbox{AR$(p)$-ARCH$(\infty)$ processes with $1\leq p \leq 8$, where the ARCH($\infty$) is defined as in Model 5} \big \}.\]
The results of simulations are featured in Table \ref{tab:4}.
\begin{center}
\captionof{table}{Percentage of selected order   based on 1000 replications depending on sample's length for model 5} \label{tab:4}
\begin{tabular}{|c|lr|lr|lr|lr|}
\toprule
\multicolumn{1}{|c|}{Sample length}& \multicolumn{2}{c}{$100$ } & \multicolumn{2}{c}{$500$ }& \multicolumn{2}{c}{$1000$ } & \multicolumn{2}{c|}{$2000$ }\\
 &  $\log n$ & $n^{2/3}$&  $\log n$ & $n^{2/3}$& $\log n$ & $n^{2/3}$& $\log n$ & $n^{2/3}$ \\
\midrule
$p<2$ &8.9 &69.1  &0.1  &17.6 &0   &6.1  &0   &0.9    \\
$p=2$ &88.9 &30.9  &75.3 &82.4 &71.4  &93.9 &72.5  &99.1  \\
$p>2$ &2.2  &0   &24.6 &0  &28.6  &0  &27.5  &0   \\

\bottomrule
\end{tabular}
\end{center}

\vspace{0.5cm}
\noindent
Note that we also computed the $\log n$ criterion for $n=5000$ and $n=10000$ in additional numerical experiments and its frequencies of choice of the true order $p=2$ were almost $72 \%$. As a consequence, for this model selection framework of the infinite memory process with a slow decrease of Lipschitz coefficients, the usual BIC penalty $\log{n}$ seems numerically not sufficient to avoid overfitting in contrast with a $n^{2/3}$ penalty that leads to a consistent criterion.
\subsection{Illustrative Example}

We consider the returns of the daily closing prices of the FTSE index of the London Stock Exchange 100. They are 2273 observations from January 4th, 2010 to December 31st, 2018. The mean and standard deviation of the  returns are -0.54 and 57.67, respectively. The Time plot and the correlograns for the log returns and squared log returns are plotted in Figure \ref{arma}.

The Figures (\ref{fig:pl1}) and (\ref{fig:pl2}) exhibit the conditional heteroskedasticity in the log return time series. Moreover, Figure (\ref{fig:acf1}) shows that more than 5 per cent of the autocorrelations are out of the confidence interval $\pm 1.96/\sqrt{2273}$ and specially the Figure (\ref{fig:acf2}) suggests that the  strong white noise assumption cannot be sustained for this log-returns sequence of FTSE index.

\begin{figure}
  \centering
  \begin{tabular}{cc}
     \subcaptionbox{Time plot of log returns.\label{fig:pl1}}[0.4\linewidth]{\includegraphics[width=6cm]{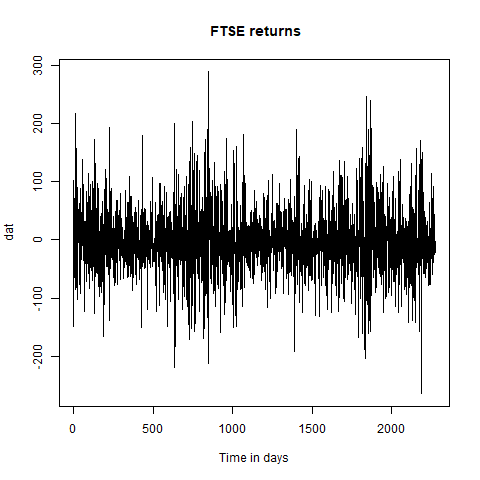}} &
    \subcaptionbox{Correlograms of log returns.\label{fig:acf1}}[0.4\linewidth]{\includegraphics[width=6cm]{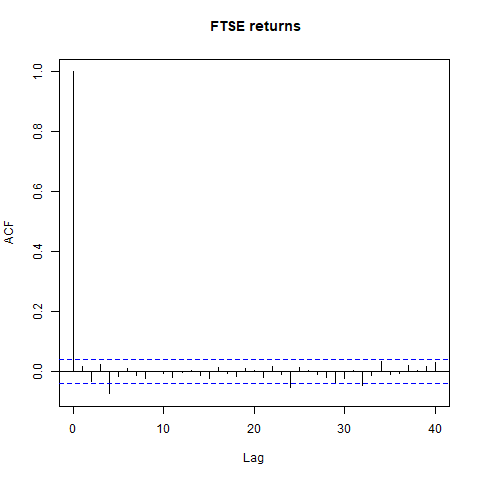}} \\
  \subcaptionbox{Time plot of squared log returns.\label{fig:pl2}}[0.4\linewidth]{\includegraphics[width=6cm]{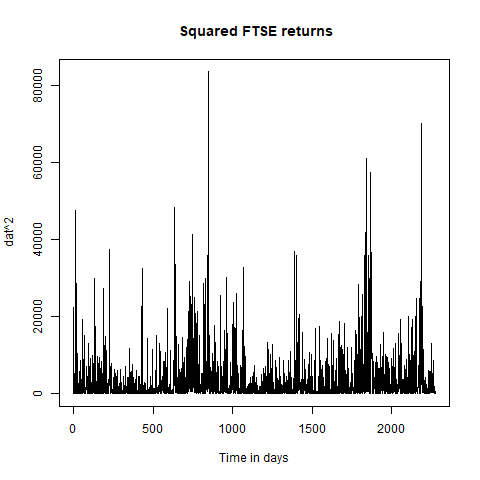}} &
    \subcaptionbox{Correlograms of squared log returns.\label{fig:acf2}}[0.4\linewidth]{\includegraphics[width=6cm]{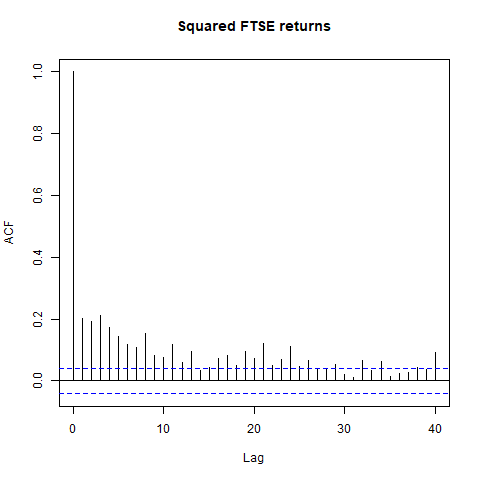}}
  \end{tabular}
  \caption{Daily closing FTSE 100 index (January 4th, 2010 to December 31 st, 2018).\label{arma}}
\end{figure}

Therefore, the $GARCH(p,q)$ family was considered for the modelling of the FTSE index with $(p,q) \in \llbracket 1; 10\rrbracket \times \llbracket 0; 10\rrbracket $ which lead us to 110 candidate models. The penalization $\log n$ and $\sqrt{n} $  have been applied to identify the best order and the goodness-of-fit of the selected model has been tested by the portmanteau test. Based on the results of the simulations, we set $K=3$ for the portmanteau test statistic.

The GARCH$(1,1)$ is the "best" model according to both criteria (related to above penalizations) and the portmanteau statistic $\widehat{Q}_3(\widehat m) \simeq 2.13$ is associated with a p-value of $0.55$.
Hence,  the selected model GARCH$(1,1)$ is adequate to model the FTSE 100 index  using either criterion.

\section{Proofs}\label{Proofs}
We start with the proof of the Proposition \ref{prop0}.
\begin{proof}
For ease of writing, consider only the general case where $f^{(i)}_{\theta_i}=g^{(i)}_{\alpha_i}$ and $M^{(i)}_{\theta_i}=N^{(i)}_{\beta_i}$ where $\theta_i={}^t(\alpha_i,\beta_i)$ for $i=1,2$. Now, assume that there exist $\alpha \in \R^{\delta}$, where $0 \leq \delta \leq \min(d_1,d_2)$ and a function $h_\alpha$ such as $g^{(1)}_{\alpha_1}=h_{\alpha}+ \ell^{(1)}_{\alpha'_1}$, $f^{(2)}_{\alpha_2}=h_{\alpha}+ \ell^{(2)}_{\alpha'_2}$ with $\alpha_1={}^t (\alpha,\alpha'_1)$ and $\alpha_2={}^t (\alpha,\alpha'_2)$ and $\ell^{(i)}_{0}=0$. \\
Similarly, assume that there exist $\beta \in \R^{\delta'}$, where $0 \leq \delta' \leq \min(d_1,d_2)$ and a function $R_\beta$ such as $N^{(1)}_{\beta_1}=R_{\beta}+ m^{(1)}_{\beta'_1}$, $N^{(2)}_{\beta_2}=R_{\beta}+ m^{(2)}_{\beta_2'}$ with $\beta_1={}^t (\beta,\beta'_1)$ and $\beta_2={}^t (\beta,\beta'_2)$ and $m^{(i)}_{0}=0$. \\
Consider now $\theta={}^t (\alpha,\alpha'_1,\alpha'_2,\beta,\beta'_1,\beta'_2) \in \R^d$ (and therefore $\max(d_1,d_2)\leq d \leq d_1+d_2$), $f_\theta=h_{\alpha}+ \ell^{(1)}_{\alpha'_1}+ \ell^{(2)}_{\alpha'_2}$  and $M_\theta=R_{\beta}+ m^{(1)}_{\beta'_1}+ m^{(2)}_{\beta'_2}$. Then if $X \in \mathcal{AC}\big(M_{\theta},f_{\theta}\big)$, \mbox{for any $t\in \Z$},
\begin{multline*}
X_t=\big (R_{\beta}((X_{t-k})_{k \geq 1})+ m^{(1)}_{\beta'_1}((X_{t-k})_{k \geq 1})+ m^{(2)}_{\beta'_2}((X_{t-k})_{k \geq 1}) \big ) \, \xi_t \\
+ \big (h_{\alpha}((X_{t-k})_{k \geq 1})+ \ell^{(1)}_{\alpha'_1}((X_{t-k})_{k \geq 1})+ \ell^{(2)}_{\alpha'_2}((X_{t-k})_{k \geq 1}) \big ).
\end{multline*}
Then, for $\alpha_2'=\beta_2'=0$, $X \in  \mathcal{AC}\big(M^{(1)}_{\theta_1},f^{(1)}_{\theta_1}\big)$ and for $\alpha_1'=\beta_1'=0$, $X \in  \mathcal{AC}\big(M^{(2)}_{\theta_2},f^{(2)}_{\theta_2}\big)$.\\
\end{proof}
\noindent In the sequel, some lemmas are stated and theirs proofs are given.
\begin{lemma} \label{lem1}
Let $X \in \mathcal{AC}(M_{\theta},f_{\theta})$ (or $\widetilde{\mathcal{AC}}(\widetilde {H}_{\theta})$) and $\Theta \subseteq \Theta(r)$  (or $\Theta \subseteq \widetilde \Theta(r)$) with {$r\ge 2$}. Assume that the assumptions $D(\Theta)$ and $K(\Theta)$ (or $\widetilde K(\Theta)$) hold. Then:
\begin{equation}
\frac{1}{\kappa_n} \, \big \|\widehat{L}_n(\theta)-L_n(\theta) \big \|_{\Theta} \limiteasn 0.
\label{eq:kounias}
\end{equation}
\label{lem:1}
\end{lemma}
\begin{proof}
We have
$|\widehat{L}_n(\theta)-L_n(\theta)|\le \sum_{t=1}^n|\widehat{q}_t(\theta)-q_t(\theta)| $. Then,
\[
\frac{1}{\kappa_n} \, \big \|\widehat{L}_n(\theta)-L_n(\theta) \big \|_{\Theta}\le \frac{1}{\kappa_n} \, \sum_{t=1}^n\|\widehat{q}_t(\theta)-q_t(\theta)\|_{\Theta}.\]
By Corollary 1 of \cite{kounias}, with $r\le 3$, \eqref{eq:kounias} is established when:
\begin{equation}
\sum_{k\ge 1}(\frac{1}{\kappa_k})^{r/3} \mathbb{E}\big(\|\widehat{q}_k(\theta)-q_k(\theta)\|_{\Theta} ^{r/3}\big) < \infty.
\end{equation}
With $r\geq 3$, and under the assumptions, we first recall some results already obtained in \cite{barW}: for any $t \in \Z$,
\begin{eqnarray}\label{recall}
&& \bullet \quad \E \big [|X_t|^r +\| f_\theta^t \|^r _{\Theta} +\| \widehat{f}_\theta^t \|^r _{\Theta} +\| M_\theta^t \|^r _{\Theta} + \| \widehat{M}_\theta^t \|^r _{\Theta} +\| H_\theta^t \|^{r/2} _{\Theta} + \| \widehat{H}_\theta^t \|^{r/2} _{\Theta} \big ] <\infty \\
\label{recall2} &&\bullet \quad \left \{ \begin{array}{l} \E \big [\| f_\theta^t -\widehat{f}_\theta^t \|^r _{\Theta}\big ]\leq  C \,\Big ( \sum_{j \ge t} \alpha_j (f_\theta,\Theta) \Big )^r\\
\E \big [\| M_\theta^t -\widehat{M}_\theta^t \|^r _{\Theta}\big ]\leq  C \,\Big (  \sum_{j \ge t} \alpha_j (M_\theta,\Theta) \Big )^r\\
\E \big [\| H_\theta^t -\widehat{H}_\theta^t \|^{r/2} _{\Theta}\big ]\leq   C \, \Big ( \min \Big \{ \sum_{j \ge t} \alpha_j (M_\theta,\Theta)~, ~ \sum_{j \ge t} \alpha_j (H_\theta,\Theta) \Big \} \Big )^{r/2}.
\end{array} \right .
\end{eqnarray}
For any $\theta \in \Theta$, we have:
\begin{align*}
&|\widehat{q}_t(\theta)-q_t(\theta)|=\Big  |\frac{(X_t-\widehat{f}_{\theta}^t)^2}{\widehat{H}_{\theta}^t}+ \log(\widehat{H}_{\theta}^t)-\frac{(X_t-f_{\theta}^t)^2}{H_{\theta}^t}-\log(H_{\theta}^t) \Big|\\
& \le (H_{\theta}^t \widehat{H}_{\theta}^t )^{-1} \big |H_{\theta}^t(X_t-\widehat{f}_{\theta}^t)^2-\widehat{H}_{\theta}^t(X_t-f_{\theta}^t)^2\big |+ \big |\log(\widehat{H}_{\theta}^t)-\log(H_{\theta}^t)\big | \\
&\le (H_{\theta}^t \widehat{H}_{\theta}^t )^{-1} \big  |(H_{\theta}^t-\widehat{H}_{\theta}^t)(X_t-f_{\theta}^t)^2 - H_{\theta}^t(X_t-f_{\theta}^t)^2+H_{\theta}^t(X_t-\widehat{f}_{\theta}^t)^2 \big |+ \big |\log(\widehat{H}_{\theta}^t)-\log(H_{\theta}^t) \big|\\
& \le \underline{h}^{-3/2} \big ( |X_t|^2+ 2|X_t\|f_{\theta}^t|+|f_{\theta}^t|^2 \big )\, \big |M_{\theta}^t-\widehat{M}_{\theta}^t\big |+\underline{h}^{-1} \big (2|X_t|+|f_{\theta}^t|+|\widehat{f}_{\theta}^t| \big ) \, \big |f_{\theta}^t-\widehat{f}_{\theta}^t \big |+ 2 \, \big  |\log(\widehat{M}_{\theta}^t)-\log(M_{\theta}^t) \big|\\
& \le \underline{h}^{-3/2}\big ( |X_t|^2+ 2|X_t| \times \|f_{\theta}^t\|_{\Theta}+\|f_{\theta}^t\|_{\Theta}^2\big ) \, \|M_{\theta}^t-\widehat{M}_{\theta}^t\|_{\Theta} \\
& \hspace{5cm} +\underline{h}^{-1} \big ( 2|X_t|+\|f_{\theta}^t\|_{\Theta}+\|\widehat{f}_{\theta}^t\|_{\Theta}\big ) \, \|f_{\theta}^t-\widehat{f}_{\theta}^t\|_{\Theta}+2 \,\underline{h}^{-1/2}  \|\widehat{M}_{\theta}^t-M_{\theta}^t\|_{\Theta}.
\end{align*}
\noindent 1/ If $X \subset {\mathcal AC}(M_\theta,f_\theta)$, we deduce
\begin{multline}
\mathbb{E}\big [ \|\widehat{q}_{t}(\theta)-q_{t}(\theta)\|_{\Theta}^{r/3} \big ]  \le C \,\Big(\mathbb{E}\Big[\big (\|X_t+f_{\theta}^t\|_{\Theta}^2 +1\big ) ^{r/3}\, \|M_{\theta}^t-\widehat{M}_{\theta}^t\|^{r/3}_{\Theta}\Big] \\
+\mathbb{E}\Big[ \big ( 2|X_t|+\|f_{\theta}^t\|_{\Theta}+ \|\widehat{f}_{\theta}^t\|_{\Theta} \big )^{r/3}\, \|f_{\theta}^t-\widehat{f}_{\theta}^t\|^{r/3}_{\Theta}\Big] \Big ).
\label{eq:ala}
\end{multline}

Then, by Hölder's inequality and \eqref{recall} we have:
\begin{multline} \label{ineg1}
\mathbb{E}\Big[\big (\|X_t+f_{\theta}^t\|_{\Theta}^2 +1\big ) ^{r/3}\, \|M_{\theta}^t-\widehat{M}_{\theta}^t\|^{r/3}_{\Theta}\Big]    \\
 \le \Big(\mathbb{E} \big [\|X_t+ f_{\theta}^t+1\|_{\Theta}^r \big ]\Big)^{2/3} \, \Big(\mathbb{E} \big [\|M_{\theta}^t-\widehat{M}_{\theta}^t\|_{\Theta}^r\big ]\Big)^{1/3}
\le C \  \Big (\mathbb{E}\big [\|M_{\theta}^t-\widehat{M}_{\theta}^t\|_{\Theta}^r \Big ]\Big)^{1/3}.
\end{multline}
Again with Hölder's inequality and \eqref{recall} ,
\begin{equation} \label{ineg2}
\mathbb{E}\big[ \big ( (2|X_t|+\|f_{\theta}^t\|_{\Theta}+\|\widehat{f}_{\theta}^t\|_{\Theta})\|f_{\theta}^t-\widehat{f}_{\theta}^t\|_{\Theta} \big )^{r/3}\big]\le C \, \big (\mathbb{E} \big [\|f_{\theta}^t-\widehat{f}_{\theta}^t\|_{\Theta}^r]\big)^{1/3}.
\end{equation}
Therefore, from \eqref{ineg1}, \eqref{ineg2} and \eqref{recall2}, there exists a constant $C$ such that
\begin{equation}
\mathbb{E} \big [ \|(\widehat{q}_{t}(\theta)-q_{t}(\theta)\|_{\Theta}^{r/3} \big ] \le C \,\Big(\sum_{j \ge t} \alpha_j (f_\theta,\Theta)+\sum_{j \ge t} \alpha_j (M_\theta,\Theta)\Big)^{r/3}.
\label{e:fond}
\end{equation}
Hence,
\[
\sum_{k\ge 1}(\frac{1}{\kappa_k})^{r/3} \mathbb{E}\big[\|\widehat{q}_k(\theta)-q_k(\theta)\|_{\Theta} ^{r/3}\big] \le C \, \sum_{k\ge 1}(\frac{1}{\kappa_k})^{r/3}\Big(\sum_{j \ge k} \alpha_j (f_\theta,\Theta)+\alpha_j (M_\theta,\Theta)\Big)^{r/3},
\]
which is finite  by assumption $K(\Theta)$, and this achieves the proof. \\

~ \\

\noindent 2/ If $X \subset \widetilde {{\mathcal AC}}(\widetilde H_\theta)$ and using Corollary 1 of \cite{kounias}, with $r\le 4$, \eqref{eq:kounias} is established when:
\begin{equation}
\sum_{k\ge 1}(\frac{1}{\kappa_k})^{r/4} \mathbb{E}\big(\|\widehat{q}_k(\theta)-q_k(\theta)\|_{\Theta} ^{r/4}\big) < \infty.
\end{equation}
By proceeding as in the previous case, we deduce
\begin{equation*}
|\widehat{q}_t(\theta)-q_t(\theta)|
\le \underline{h}^{-2} |X_t|^2 \, \|H_{\theta}^t-\widehat{H}_{\theta}^t\|_{\Theta} +\underline{h}^{-1}  \|\widehat{H}_{\theta}^t-H_{\theta}^t\|_{\Theta}.
\end{equation*}
In addition, we deduce that  there exists a constant $C$ such that
\begin{equation}
\mathbb{E} \big [ \|(\widehat{q}_{t}(\theta)-q_{t}(\theta)\|_{\Theta}^{r/4} \big ] \le C \,\Big(\sum_{j \ge t} \alpha_j (H_\theta,\Theta)\Big)^{r/4}.
\label{e:fond2}
\end{equation}
\end{proof}

\begin{lemma} \label{lem2}
Let $X \in \mathcal{AC}(M_{\theta},f_{\theta})$ (or $\widetilde{\mathcal{AC}}(\widetilde {H}_{\theta})$) and $\Theta \subseteq \Theta(r)$  (or $\Theta \subseteq \widetilde \Theta(r)$) with {$r\ge 2$}. Assume that the assumptions $D(\Theta)$ and $K(\Theta)$ (or $\widetilde K(\Theta)$) hold. Then:
\begin{equation}
\frac{1}{\kappa_n} \,\Big \|\dfrac{\partial\widehat{L}_n(\theta)}{\partial \theta}-\frac{\partial L_n(\theta)}{\partial \theta}\Big \|_{\Theta}\limiteasn 0.
\label{eq:2}
\end{equation}
\label{lem:2}
\end{lemma}

\begin{proof}
We will go along similar lines as in the proof of Lemma \ref{lem:1}. We have:
\[
\frac{1}{\kappa_n} \, \Big \|\dfrac{\partial\widehat{L}_n(\theta)}{\partial \theta}-\frac{\partial L_n(\theta)}{\partial \theta}\Big  \|_{\Theta}\le \frac{1}{\kappa_n} \, \sum_{t=1}^n \Big \| \dfrac{\partial \widehat{q}_t(\theta)}{\partial \theta_i}-\frac{\partial q_t(\theta)}{\partial \theta_i}\Big \|_{\Theta}. \]
Using again Corollary 1 of \cite{kounias}, it is sufficient to prove for $r\le 3$ that
\begin{equation}
\sum_{k\ge 1}(\frac{1}{\kappa_k})^{r/3}  \,\mathbb{E} \Big [\Big \| \dfrac{\partial \widehat{q}_t(\theta)}{\partial \theta_i}-\frac{\partial q_t(\theta)}{\partial \theta_i}\Big \|_{\Theta}^{r/3} \Big ] < \infty.
\label{eq:kou}
\end{equation}
For any $\theta \in \Theta$, with $H_\theta=M_\theta^2$, the first partial derivatives of $q_t(\theta)$ are
\begin{multline*}
\frac{\partial q_t(\theta)}{\partial\theta_i}= \frac{-2(X_t-f_{\theta}^t)}{H_{\theta}^t}\frac{\partial f_{\theta}^t}{\partial \theta_i}-
\frac{(X_t-f_{\theta}^t)^2}{(H_{\theta}^t)^2}\frac{\partial H_{\theta}^t}{\partial \theta_i}+\frac{1}{H_{\theta}^t}\frac{\partial H_{\theta}^t}{\partial \theta_i}\\
=-2 (H_{\theta}^t)^{-1}(X_t-f_{\theta}^t)\frac{\partial f_{\theta}^t}{\partial \theta_i}+(X_t-f_{\theta}^t)^2\frac{\partial (H_{\theta}^t)^{-1}}{\partial \theta_i}+(H_{\theta}^t)^{-1}\frac{\partial H_{\theta}^t}{\partial \theta_i},
\end{multline*}
for $i=1,\cdots,d$.
Hence,
\begin{multline*}
\Big |\frac{\partial \widehat{q}_t(\theta)}{\partial \theta_i}-\frac{\partial q_t(\theta)}{\partial \theta_i} \Big |\le 2 \,  \Big |(h_{\theta}^t)^{-1}(X_t-f_{\theta}^t)\frac{\partial f_{\theta}^t}{\partial \theta_i}-(\widehat{h}_{\theta}^t)^{-1}(X_t-\widehat{f}_{\theta}^t)\frac{\partial \widehat{f}_{\theta}^t}{\partial \theta_i} \Big |\\
+ \Big |(X_t-\widehat{f}_{\theta}^t)^2\frac{\partial (\widehat{H}_{\theta}^t)^{-1}}{\partial \theta_i}-(X_t-f_{\theta}^t)^2\frac{\partial (H_{\theta}^t)^{-1}}{\partial \theta_i} \Big | + \Big |(\widehat{H}_{\theta}^t)^{-1}\frac{\partial \widehat{H}_{\theta}^t}{\partial \theta_i}-(H_{\theta}^t)^{-1}\frac{\partial H_{\theta}^t}{\partial \theta_i} \Big |.
\end{multline*}
Then, using
$|a_1b_1c_1-a_2b_2c_2|\le |a_1-a_2| \,|b_2| \,|c_2|+|a_1|\, |b_1-b_2|\,|c_2|+|a_1|\,|b_1|\,|c_1-c_2|$ for any $a_1,a_2,b_1,b_2,c_1,c_2$ in $\R$, we obtain
\begin{align*}
\Big |\frac{\partial \widehat{q}_t(\theta)}{\partial \theta_i}-\frac{\partial q_t(\theta)}{\partial \theta_i} \Big |& \le 2 \,\Big ( \big |(H_{\theta}^t)^{-1}-(\widehat{H}_{\theta}^t)^{-1} \big | \times \big |X_t-\widehat{f}_{\theta}^t\big | \, \Big |\frac{\partial \widehat{f}_{\theta}^t}{\partial \theta_i}\Big|+\big|(H_{\theta}^t)^{-1}\big | \times \big | \widehat{f}_{\theta}^t -f_{\theta}^t\big | \, \Big | \frac{\partial \widehat{f}_{\theta}^t}{\partial \theta_i}\Big|\\
&+\big |(H_{\theta}^t)^{-1} \big | \times \big |X_t-f_{\theta}^t\big | \, \Big | \frac{\partial f_{\theta}^t}{\partial \theta_i}-\frac{\partial \widehat{f}_{\theta}^t}{\partial \theta_i} \Big | \Big )+ \big |X_t-\widehat{f}_{\theta}^t \big |^2 \, \Big|\frac{\partial (\widehat{H}_{\theta}^t)^{-1}}{\partial \theta_i}-\frac{\partial (H_{\theta}^t)^{-1}}{\partial \theta_i}\Big|\\
&+ 2 \,\Big |\frac{\partial (H_{\theta}^t)^{-1}}{\partial \theta_i} \Big | \, \big | X_t \big | \, \big |f_{\theta}^t-\widehat{f}_{\theta}^t\big |+\big |(\widehat{H}_{\theta}^t)^{-1} \big | \, \Big | \frac{\partial \widehat{H}_{\theta}^t}{\partial \theta_i}-\frac{\partial H_{\theta}^t}{\partial \theta_i}\Big|+ \Big |\frac{\partial H_{\theta}^t}{\partial \theta_i} \Big | \, \big |(\widehat{H}_{\theta}^t)^{-1}-(H_{\theta}^t)^{-1} \big |.
\end{align*}
Thus,
\begin{align*}
&\Big \| \dfrac{\partial \widehat{q}_t(\theta)}{\partial \theta_i}-\frac{\partial q_t(\theta)}{\partial \theta_i}\Big \|_{\Theta} \le 2\, \underline{h}^{-1}\Big (\big \|\widehat{f}_{\theta}^t -f_{\theta}^t \big \|_{\Theta} \Big \|\dfrac{\partial \widehat{f}_{\theta}^t}{\partial \theta_i}\Big \|_{\Theta}+\big \|X_t-f_{\theta}^t\big  \|_{\Theta} \Big \|\dfrac{\partial f_{\theta}^t}{\partial \theta_i}-\frac{\partial \widehat{f}_{\theta}^t}{\partial \theta_i} \Big \|_{\Theta}  \Big )\\
&\hspace{0.5cm} + 2\, \big \|(H_{\theta}^t)^{-1}-(\widehat{H}_{\theta}^t)^{-1}\big  \|_{\Theta} \big  \|X_t-\widehat{f}_{\theta}^t \big \|_{\Theta}\Big \|\dfrac{\partial \widehat{f}_{\theta}^t}{\partial \theta_i}\Big \|_{\Theta}+
\big  \|X_t-\widehat{f}_{\theta}^t\big  \|^2 \Big \|\dfrac{\partial (\widehat{H}_{\theta}^t)^{-1}}{\partial \theta_i}-\frac{\partial (H_{\theta}^t)^{-1}}{\partial \theta_i}\Big \|\\
&\hspace{0.5cm} + 2 \,\big |X_t\big | \,\big \|f_{\theta}^t-\widehat{f}_{\theta}^t \big \|_{\Theta}  \Big \|\frac{\partial (H_{\theta}^t)^{-1}}{\partial \theta_i}\Big \|_{\Theta}+
\big  \|(\widehat{H}_{\theta}^t)^{-1}\big  \|_{\Theta} \Big \|\dfrac{\partial \widehat{H}_{\theta}^t}{\partial \theta_i}-\frac{\partial H_{\theta}^t}{\partial \theta_i}\Big \|_{\Theta}+\big \|(\widehat{H}_{\theta}^t)^{-1}-(H_{\theta}^t)^{-1}\|_{\Theta} \Big \|\frac{\partial H_{\theta}^t}{\partial \theta_i}\Big \|_{\Theta} .
\end{align*}
Using again the results of \cite{barW}, we know that:
\begin{eqnarray}\label{recall3}
&& \hspace{-1.5cm}\bullet \quad \E \Big [\Big \|\frac{\partial f_{\theta}^t}{\partial \theta_i}\Big \|_{\Theta}^r\hspace{-2mm} +\Big \|\frac{\partial \widehat{f}_{\theta}^t}{\partial \theta_i}\Big \|_{\Theta}^r\hspace{-2mm}+\Big \|\frac{\partial M_{\theta}^t}{\partial \theta_i}\Big \|_{\Theta}^r \hspace{-2mm}+\Big \|\frac{\partial \widehat{M}_{\theta}^t}{\partial \theta_i}\Big \|_{\Theta}^r\hspace{-2mm} +\Big \|\frac{\partial H_{\theta}^t}{\partial \theta_i}\Big \|_{\Theta}^{r/2}\hspace{-2mm} +\Big \|\frac{\partial (H_{\theta}^t)^{-1}}{\partial \theta_i}\Big \|_{\Theta}^{r}\Big ] <\infty \\
\label{recall4} && \hspace{-1.5cm}\bullet \quad \left \{ \begin{array}{l} \displaystyle  \E \big [ \big \|  (H_{\theta}^t)^{-1}-(\widehat{H}_{\theta}^t)^{-1} \big \|^r _{\Theta}\big ]\leq  C \,\Big ( \sum_{j \ge t} \alpha_j (M_\theta,\Theta) \Big )^r\\
\displaystyle  \E \Big [ \Big \| \frac{\partial f_{\theta}^t}{\partial \theta_i}-\frac{\partial \widehat{f}_{\theta}^t}{\partial \theta_i} \Big \|^r _{\Theta}\Big ]\leq  C \,\Big ( \sum_{j \ge t} \alpha_j (\partial f_\theta,\Theta) \Big )^r\\
 \displaystyle  \E \Big [ \Big \| \frac{\partial H_{\theta}^t}{\partial \theta_i}-\frac{\partial \widehat{H}_{\theta}^t}{\partial \theta_i} \Big \|^{r/2} _{\Theta}\Big ]\leq  C \,\Big (  \sum_{j \ge t} \big (\alpha_j (M_\theta,\Theta)+ \alpha_j (\partial M_\theta,\Theta) \big )\Big )^{r/2}\\
\displaystyle  \E \Big [ \Big \| \frac{\partial (H_{\theta}^t)^{-1}}{\partial \theta_i}-\frac{\partial (\widehat{H}_{\theta}^t)^{-1}}{\partial \theta_i} \Big \|^{r/2} _{\Theta}\Big ]\leq  C \,\Big (  \sum_{j \ge t} \big (\alpha_j (M_\theta,\Theta)+\alpha_j (\partial M_\theta,\Theta) \big )\Big )^{r/2}
\end{array} \right .
\end{eqnarray}
1. If $X \subset {\mathcal AC}(M_\theta,f_\theta)$, we deduce from the Hölder's Inequality that,
\begin{align*}
&\mathbb{E}\Big [ \Big \| \dfrac{\partial \widehat{q}_t(\theta)}{\partial \theta_i}-\frac{\partial q_t(\theta)}{\partial \theta_i}\Big \|_{\Theta}^{r/3} \Big ] \le C\, \Big [ \big ( \mathbb{E} \big [ \big \|\widehat{f}_{\theta}^t -f_{\theta}^t \big \|_{\Theta}^{r} \big ] \big )^{1/3} \Big (\mathbb{E}\Big [ \Big \|\dfrac{\partial \widehat{f}_{\theta}^t}{\partial \theta_i}\Big \|_{\Theta}^{r/2}\Big ] \Big )^{2/3}\\
&+\big ( \mathbb{E} \big [ \big \|X_t -f_{\theta}^t \big \|_{\Theta}^{2r/3} \big ] \big )^{1/2} \Big (\mathbb{E}\Big [ \Big \|\dfrac{\partial f_{\theta}^t}{\partial \theta_i}-\frac{\partial \widehat{f}_{\theta}^t}{\partial \theta_i} \Big \|_{\Theta}^{r}\Big ] \Big )^{1/3} \\
&+
\big (\mathbb{E}\big [ \big \|(H_{\theta}^t)^{-1}-(\widehat{H}_{\theta}^t)^{-1} \big\|_{\Theta}^{r} \big ] \big )^{1/3}\,\Big (  \mathbb{E} \big [ \big \|X_t-\widehat{f}_{\theta}^t \big \|_{\Theta}^{r} \big ]\, \mathbb{E}\Big [ \Big \|\dfrac{\partial \widehat{f}_{\theta}^t}{\partial \theta_i}\Big \|_{\Theta}^{r} \Big ]\Big)^{1/3}\\
&+
\big (\mathbb{E} \big [ \big \|X_t-\widehat{f}_{\theta}^t \big \|_{\Theta}^{r} \big ] \Big )^{1/3} \Big ( \mathbb{E} \Big [\Big \|\dfrac{\partial (\widehat{H}_{\theta}^t)^{-1}}{\partial \theta_i}-\frac{\partial (H_{\theta}^t)^{-1}}{\partial \theta_i}\Big \|^{r/2} \Big ] \Big )^{2/3} \\
& +\Big(\mathbb{E}\Big [\Big \|\frac{\partial (H_{\theta}^t)^{-1}}{\partial \theta_i}\Big \|_{\Theta}^{r} \Big ]\Big)^{1/3}\hspace{-1mm} \Big (\mathbb{E} \big [|X_t|^{r} \big ] \, \mathbb{E} \big [\big \|f_{\theta}^t-\widehat{f}_{\theta}^t \big \|_{\Theta}^{r} \big ]\Big)^{1/3}\\
&+ \Big (\mathbb{E}\Big [\Big \|\dfrac{\partial \widehat{H}_{\theta}^t}{\partial \theta_i}-\frac{\partial H_{\theta}^t}{\partial \theta_i}\Big \|_{\Theta}^{r/3} \Big ] + \Big(\mathbb{E}\Big [ \Big \|\frac{\partial H_{\theta}^t}{\partial \theta_i}\Big \|_{\Theta}^{r/2} \Big ] \Big )^{2/3} \big (\mathbb{E} \big [ \big \|(\widehat{H}_{\theta}^t)^{-1}-(H_{\theta}^t)^{-1} \big \|_{\Theta}^{r} \big ] \Big )^{1/3}\Big ].
\end{align*}

Using \eqref{recall3} and \eqref{recall4}, we deduce
\begin{multline*}
\mathbb{E}\Big [ \Big \| \dfrac{\partial \widehat{q}_t(\theta)}{\partial \theta_i}-\frac{\partial q_t(\theta)}{\partial \theta_i}\Big \|_{\Theta}^{r/3} \Big ] \le C\, \Big(\sum_{j \ge t} \alpha_j (f_\theta,\Theta)+ \alpha_j (M_\theta,\Theta)+  \alpha_j (\partial  f_\theta,\Theta)+ \alpha_j (\partial  M_\theta,\Theta)\Big)^{r/3}.
\end{multline*}
Therefore,
\begin{multline*}
\sum_{k\ge 1}\frac{1}{\kappa^{r/3}_k}\,  \mathbb{E}\Big [ \Big \| \dfrac{\partial \widehat{q}_k(\theta)}{\partial \theta_i}-\frac{\partial q_k(\theta)}{\partial \theta_i}\Big \|_{\Theta}^{r/3} \Big ] \\
 \le C \, \sum_{k\ge 1}\frac{1}{\kappa^{r/3} _k}\Big(\sum_{j \ge t} \alpha_j (f_\theta,\Theta)+ \alpha_j (M_\theta,\Theta)+  \alpha_j (\partial  f_\theta,\Theta)+ \alpha_j (\partial  M_\theta,\Theta)\Big)^{r/3}.
\label{eq:lem22}
\end{multline*}
We conclude the proof of \eqref{eq:kou} from assumption $K(\Theta)$.\\
~\\
2. If $X \subset \widetilde {{\mathcal AC}}(\widetilde H_\theta)$, we deduce
\begin{multline*}
\Big \| \dfrac{\partial \widehat{q}_t(\theta)}{\partial \theta_i}-\frac{\partial q_t(\theta)}{\partial \theta_i}\Big \|_{\Theta} \le
\big  |X_t\big  |^2 \, \Big \|\dfrac{\partial (\widehat{H}_{\theta}^t)^{-1}}{\partial \theta_i}-\frac{\partial (H_{\theta}^t)^{-1}}{\partial \theta_i}\Big \|_{\Theta} \\+
 \underline{h}^{-1}\Big \|\dfrac{\partial \widehat{H}_{\theta}^t}{\partial \theta_i}-\frac{\partial H_{\theta}^t}{\partial \theta_i}\Big \|_{\Theta}+\big \|(\widehat{H}_{\theta}^t)^{-1}-(H_{\theta}^t)^{-1}\|_{\Theta} \Big \|\frac{\partial H_{\theta}^t}{\partial \theta_i}\Big \|_{\Theta} .
\end{multline*}
As a consequence,
\begin{multline*}
\E \Big [ \Big \| \dfrac{\partial \widehat{q}_t(\theta)}{\partial \theta_i}-\frac{\partial q_t(\theta)}{\partial \theta_i}\Big \|_{\Theta} ^{r/4} \Big ]\le
\Big ( \E \big [\big  |X_t\big  |^r \, \E \Big [ \Big \|\dfrac{\partial (\widehat{H}_{\theta}^t)^{-1}}{\partial \theta_i}-\frac{\partial (H_{\theta}^t)^{-1}}{\partial \theta_i}\Big \|^{r/2}_{\Theta} \Big ] \Big )^{1/2} \\
+ \underline{h}^{-r/4}\E \Big [  \Big \|\dfrac{\partial \widehat{H}_{\theta}^t}{\partial \theta_i}-\frac{\partial H_{\theta}^t}{\partial \theta_i}\Big \|^{r/4}_{\Theta} \Big ] + \Big ( \E \big  [ \big \|(\widehat{H}_{\theta}^t)^{-1}-(H_{\theta}^t)^{-1}\|^{r/2}_{\Theta}  \big ] \, \E \Big [  \Big \|\frac{\partial H_{\theta}^t}{\partial \theta_i}\Big \|^{r/2}_{\Theta} \Big ] \Big )^{1/2},
\end{multline*}
implying
\begin{equation*}
\E \Big [ \Big \| \dfrac{\partial \widehat{q}_t(\theta)}{\partial \theta_i}-\frac{\partial q_t(\theta)}{\partial \theta_i}\Big \|_{\Theta} ^{r/4} \Big ]\le  C \, \Big ( \sum_{j \ge t} \alpha_j (H_\theta,\Theta)+  \alpha_j (\partial  H_\theta,\Theta)\Big)^{r/4},
\end{equation*}
which achieves the proof, according to Corollary 1 of \cite{kounias}.
\end{proof}

 \medskip

\begin{lemma} \label{lem3}
Under the assumptions of Theorem \ref{theo:1} and if a model $m\in\mathcal{M}$ is such that  $\theta^*\in \Theta(m)$, then:
\begin{equation}
\frac{1}{\kappa_n} \,\big |\widehat{L}_n(\widehat{\theta}(m))-\widehat{L}_n(\theta^*) \big |=o_P(1).
\end{equation}
\label{lem:3}
\end{lemma}
\begin{proof}
We have:
\begin{align*}
\frac{1}{\kappa_n} \, \big |\widehat{L}_n(\widehat{\theta}(m))-\widehat{L}_n(\theta^*)\big |&= \frac{1}{\kappa_n}\, \big|\widehat{L}_n(\widehat{\theta}(m))- L_n(\widehat{\theta}(m))+L_n(\widehat{\theta}(m))- L_n(\theta^*) +L_n(\theta^*)-\widehat{L}_n(\theta^*)\big|\\
&\le  \frac{2}{\kappa_n} \, \big \|\widehat{L}_n(\theta)-L_n(\theta) \big \|_{\Theta(r)} + \frac{1}{\kappa_n} \, \big |L_n(\widehat{\theta}(m))-L_n(\theta^*)\big |.
\end{align*}
According to Lemma \ref{lem:1}, $\frac{1}{\kappa_n}\, \big \|\widehat{L}_n(\theta)-L_n(\theta) \big \|_{\Theta(r)}\limiteasn 0$. The proof will be achieved if we prove
\begin{equation}
\frac{1}{\kappa_n} \, \big |L_n(\widehat{\theta}(m))-L_n(\theta^*) \big |=o_P(1).
\label{eq:b}
\end{equation}
Applying a second order Taylor expansion of $L_n$ around $\widehat{\theta}(m)$ for $n$ sufficiently large such that $\overline{\theta}(m) \in \Theta(m)$ which are between $\widehat{\theta}(m)$ and $\theta^*$, yields:
\begin{multline} \label{dtheta}
\frac{1}{\kappa_n}\big(L_n(\widehat{\theta}(m))-L_n(\theta^*)\big)=\\
\frac{1}{\kappa_n}\big(\widehat{\theta}(m)-\theta^*\big) \frac{\partial L_n(\widehat{\theta}(m))}{\partial \theta}+\frac{1}{2\kappa_n}\big(\widehat{\theta}(m)-\theta^*\big)'\, \frac{\partial^2 L_n(\overline{\theta}(m))}{\partial \theta^2} \, \big(\widehat{\theta}(m)-\theta^*\big).
\, \end{multline}
Let us deal first with the first term on the right hand side of last equality:
\[
\frac{1}{\kappa_n}\, \big (\widehat{\theta}(m)-\theta^* \big ) \,  \frac{\partial L_n(\widehat{\theta}(m))}{\partial \theta}=\frac{1}{\kappa_n} \, \sqrt{n} \big (\widehat{\theta}(m)-\theta^* \big ) \, \frac{1}{\sqrt{n}} \, \frac{\partial L_n(\widehat{\theta}(m))}{\partial \theta}.
\]
Since $\frac{1}{\kappa_n}=o(1)$ and from \cite{barW} we have $\sqrt{n}\big (\widehat{\theta}(m)-\theta^*\big )=O_P(1)$ and $\frac{1}{\sqrt{n}} \, \frac{\partial L_n(\widehat{\theta}(m))}{\partial \theta}=o_P(1)$, it follows that:
\begin{equation} \label{dL1}
\frac{1}{\kappa_n} \, \big (\widehat{\theta}(m)-\theta^* \big ) \frac{\partial L_n(\widehat{\theta}(m))}{\partial \theta}=o_P(1).
\end{equation}
On the other hand, for the second term of the right hand side of equality \eqref{dtheta}, let us note that, we have from \cite{barW}:
\begin{itemize}
\item $\sqrt{n} \, \big (\widehat{\theta}(m)-\theta^* \big )\limiteloin  \mathcal{A}_{\theta^*,m}$ a Gaussian random variable from \eqref{tlcqmle}.
\item $\displaystyle -\frac{2}{n} \,\Big ( \frac{\partial^2 L_n(\overline{\theta}(m))}{\partial \theta_i \partial \theta_j} \Big )_{i,j \in m} \limiteasn F({\theta}^*,m)~$ since $\widehat{\theta}(m) \limiteasn \theta^*$ and using the assumption Var($\Theta$) insuring that the matrix $F(\theta^*,m)$ exists and is definite positive (see \cite{barW}).
\end{itemize}
Hence,
\begin{multline*}
\big(\widehat{\theta}(m)-\theta^*\big)'  \,\Big ( \frac{\partial^2 L_n(\overline{\theta}(m))}{\partial \theta_i \partial \theta_j} \Big )_{i,j \in m}(\widehat{\theta}(m)-\theta^*)\\
=\frac{-1}{2}\sqrt{n}\big (\widehat{\theta}(m)-\theta^*\big )'\,\big (F(\theta^*,m)+o_P(1) \big ) \, \sqrt{n}\big (\widehat{\theta}(m)-\theta^* \big )\\
\limiteproban \frac{-1}{2}\,\mathcal{A}_{\theta^*,m}'\, F(\theta^*,m) \,\mathcal{A}_{\theta^*,m}.
\end{multline*}
We deduce that
\begin{multline}
\big(\widehat{\theta}(m)-\theta^*\big)'  \,\Big ( \frac{\partial^2 L_n(\overline{\theta}(m))}{\partial \theta_i \partial \theta_j} \Big )_{i,j \in m}(\widehat{\theta}(m)-\theta^*)=O_P(1) \\
\Longrightarrow \quad \frac{1}{\kappa_n} \, \big(\widehat{\theta}(m)-\theta^*\big)'  \,\Big ( \frac{\partial^2 L_n(\overline{\theta}(m))}{\partial \theta_i \partial \theta_j} \Big )_{i,j \in m}(\widehat{\theta}(m)-\theta^*)=o_P(1).
\label{eq:xx}
\end{multline}
 Thus,  \eqref{eq:b} follows from \eqref{dtheta}, \eqref{dL1} and \eqref{eq:xx}; which completes the proof of Lemma \ref{lem:3}.
\end{proof}

\subsection{Misspecified model}
When a model $m$ is misspecified, we will show that $\mathbb{P}(\widehat{m}=m^*) \limiten 0$ following the same scheme of proof than in \cite{white2}. Before dealing with this proof, we state some useful results.
\begin{proposition}\label{prop:1}
Let $X \in \mathcal{AC}(M_{\theta},f_{\theta})$ (or $\widetilde{\mathcal{AC}}(\widetilde {H}_{\theta})$) and $\Theta \subseteq \Theta(r)$  (or $\Theta \subseteq \widetilde \Theta(r)$) with $r\ge 2$. Then, when the assumption $D(\Theta)$ holds,
\begin{equation}\label{defL}
\Big \| \frac 1 n \,  L_n(\theta)-L(\theta) \Big \|_{\Theta} \limiteasn 0\quad \mbox{with}\quad L(\theta):=-\frac 1 2 \,\E [q_0(\theta)].
\end{equation}
\end{proposition}
\begin{proof}
See the proof of Theorem 1 in \cite{barW}.
\end{proof}

\begin{lemma}\label{lem:5}
Under the assumptions of Theorem \ref{theo:1} and for $m \in \mathcal{M}$ such as $m^*\subset m$, then:
\begin{equation}
L_n(\widehat{\theta}(m))-L_n(\theta^*)=O_P(1).
\label{eq:lem1}
\end{equation}
\end{lemma}
\begin{proof}
Applying a second order Taylor expansion of $L_n$ around $\widehat{\theta}(m^*)$ for $n$ sufficiently large such that $\overline{\theta}(m) \in \Theta(m)$ which are between $\theta^*$ and $\widehat{\theta}(m^*)$, yields:
\begin{eqnarray*}
L_n(\widehat{\theta}(m))-L_n(\theta^*)& =&(\widehat{\theta}(m)-\theta^*) \frac{\partial L_n(\widehat{\theta}(m))}{\partial \theta}+\frac{1}{2}(\widehat{\theta}(m)-\theta^*)'\frac{\partial^2 L_n(\overline{\theta}(m))}{\partial \theta \partial \theta'}(\widehat{\theta}(m)-\theta^*)\\
& &\hspace{-2cm}= \sqrt{n}(\widehat{\theta}(m)-\theta^*)\frac{1}{\sqrt{n}}\frac{\partial L_n(\widehat{\theta}(m))}{\partial \theta}+\frac{1}{2}\sqrt{n}(\widehat{\theta}(m)-\theta^*)'\frac{1}{n}\frac{\partial^2 L_n(\overline{\theta}(m))}{\partial \theta \partial \theta'}\sqrt{n}(\widehat{\theta}(m)-\theta^*)\\
& &\hspace{-2cm}= \hspace{1.5cm} o_p(1) \hspace{3cm} +\hspace{1.5cm}O_P(1) \hspace{1.5cm}\\
& &\hspace{-2cm}=  O_P(1),
\end{eqnarray*}
by using equality \eqref{eq:xx}.
\end{proof}

\subsection{Proof of Theorem \ref{theo:1}}
As we point out in Subsection \ref{qmle_msc}, the proof is divided into two parts.
\begin{proof}
1. For $m\in {\cal M}$ such as $m*\subset m$ and $m\neq m^*$ (overfitting), then using with $\widehat{C}(m)=-2\widehat{L}_n\big(\widehat{\theta}(m)\big)+ |m| \,\kappa_n$ (see \eqref{eq:cri}), we have:
\begin{eqnarray*}
\P(\widehat{m}=m) & \le & \P\big( \widehat{C}(m) \le -2\widehat{L}_n\big(\theta^*\big)+ |m^*| \,\kappa_n\big)\\
& \le &\P\Big(-2\big(\widehat{L}_n(\widehat{\theta})-\widehat{L}_n(\theta^*)\big) \le \kappa_n(|m^*|-|m|)\Big)\\
& \le &\P\Big(\frac{1}{\kappa_n}\big(\widehat{L}_n(\theta^*)-\widehat{L}_n(\widehat{\theta})\big) \le \frac{(|m^*|-|m|)}{2}\Big)\\
& \limiten & 0 \;
\end{eqnarray*}
\textnormal{by virtue of Lemma}\; \ref{lem:3} and because $|m|-|m^*|\geq 1 $. \\
~\\
2. Let $m\in {\cal M}$ such as $m^* \not \subset m$. Then,
\begin{multline}
\hspace{-0.5cm} \widehat{L}_n(\widehat{\theta}(m^*))-\widehat{L}_n(\widehat{\theta}(m))=\big(\widehat{L}_n(\widehat{\theta}(m^*))-L_n(\widehat{\theta}(m^*))\big)-\big(\widehat{L}_n(\widehat{\theta}(m))-L_n(\widehat{\theta}(m))  \big) \\
+ \big(L_n(\widehat{\theta}(m^*))- L_n(\widehat{\theta}(m)) \big).
\label{eq:cr}
\end{multline}
It follows from Lemma \ref{lem:1} that the first and the second term of the right part of \eqref{eq:cr} are equal to $o_P(\kappa_n)$.
Moreover, the third term can be written as follows:
\[
L_n(\widehat{\theta}(m^*))-L_n(\widehat{\theta}(m)) =\big(L_n(\widehat{\theta}(m^*))-L_n(\theta^*)\big)+ \big(L_n(\theta^*)-L_n(\widehat{\theta}(m)) \big).
\]
From Lemma \ref{lem:5}, one deduces $L_n(\widehat{\theta}(m^*))-L_n(\theta^*)=O_P(1)$. In addition, in the sequel, we are going to show that
\begin{equation}
L_n(\theta^*)-L_n(\widehat{\theta}(m))=n \,\big ( A(m)+o_P(1) \big ), \quad\mbox{with $A(m)>0$}.
\label{eq:mes}
\end{equation}
For any $\theta \in \Theta(m)$, we have from Proposition \ref{prop:1}
\begin{align*}
L_n(\theta^*)-L_n(\theta)& =\big(L_n(\theta^*)-n\, L(\theta^*)\big)-\big(L_n(\theta)-n\,L(\theta)\big)+n\,\big(L(\theta^*))-L(\theta))\big)\\
&=o_P(n)+ n\, \big (L(\theta^*))-L(\theta) \big).
\end{align*}
Let us denote by $\mathcal{F}_t:=\sigma \big (X_{t-1},X_{t-2},\cdots \big )$. Using conditional expectation, we obtain
\begin{equation}\label{cond}
L(\theta^*)-L(\theta)=-\frac{1}{2} \, \E \Big[ \E \big [q_0(\theta)-q_0(\theta^*)~|~\mathcal{F}_0 \big ]\Big].
\end{equation}
But,
\begin{eqnarray*}
 \E \big [q_0(\theta)-q_0(\theta^*)~|~\mathcal{F}_0 \big ]& =&
\E\Big [\frac{(X_0 - f_{\theta}^0)^2}{H_{\theta}^0} + \log(H_{\theta}^0) -\frac{(X_0 - f_{\theta^*}^0)^2}{H_{\theta^*}^0} - \log(H_{\theta^*}^0) ~ |~\mathcal{F}_0 \Big ]\\
&=& \log \Big (\frac{H_{\theta}^0}{H_{\theta^*}^0}\Big )+\frac{\E\big [(X_0 - f_{\theta}^0)^2~|~\mathcal{F}_0\big ]}{H_{\theta}^0} -\frac{\E\big [(X_0 - f_{\theta^*}^0)^2~|~\mathcal{F}_0\big ]}{H_{\theta^*}^0}\\
&=&\log\Big (\frac{H_{\theta}^0}{H_{\theta^*}^0}\Big )-1+ \frac{\E\big [(X_0 -f_{\theta^*}^0+f_{\theta^*}^0- f_{\theta}^0)^2~ |~\mathcal{F}_0\big ]}{H_{\theta}^0} \\
&=&\frac{H_{\theta^*}^0}{H_{\theta}^0}- \log\Big (\frac{H_{\theta^*}^0}{H_{\theta}^0}\Big )-1 + \frac{ (f_{\theta^*}^0-f_{\theta}^0)^2}{H_{\theta}^0}
\end{eqnarray*}
As a consequence, from \eqref{cond},
\begin{eqnarray*}
A(m)&:=&2 \, \big (L(\theta^*)-L(\theta) \big ) \\
& =& \E \Big [ \frac{H_{\theta^*}^0}{H_{\theta}^0}- \log\Big (\frac{H_{\theta^*}^0}{H_{\theta}^0}\Big )-1 + \frac{ (f_{\theta^*}^0-f_{\theta}^0)^2}{H_{\theta}^0} \Big ]\\
& \ge & \E \Big [ \frac{H_{\theta^*}^0}{H_{\theta}^0}\Big ]- \log\Big (\E \Big [ \frac{H_{\theta^*}^0}{H_{\theta}^0}\Big ]\Big )-1 +\E \Big [  \frac{ (f_{\theta^*}^0-f_{\theta}^0)^2}{H_{\theta}^0} \Big ]\quad \mbox{by Jensen Inequality.}
\end{eqnarray*}
Since $ x-\log(x)- 1 > 0$ for any $x>0,~x\neq 1$ and $ x-\log(x)- 1= 0$ for $x=1$, we deduce that
\begin{itemize}
\item If $f_{\theta^*}^0 \ne f_{\theta}^0$ then $\E \Big [  \frac{ (f_{\theta^*}^0-f_{\theta}^0)^2}{H_{\theta}^0} \Big ]>0$ and $A(m)>0.$
\item Otherwise, if $f_{\theta^*}^0 = f_{\theta}^0$, then
\begin{equation*}
A(m) =\E \Big [ \frac{H_{\theta^*}^0}{H_{\theta}^0}- \log\Big (\frac{H_{\theta^*}^0}{H_{\theta}^0}\Big )-1  \Big ],
\end{equation*}
From Assumption ID$(\Theta)$, when $\theta^*\notin \Theta(m)$ and if $f_{\theta^*}^0=f_{\theta}^0$, we necessarily have $H_{\theta^*}^0\ne H_{\theta}^0$ so that $\frac{H_{\theta^*}^0}{H_{\theta}^0}\ne 1.$ Then $A(m)>0$.
\end{itemize}
Therefore  $A(m) >0$ for any $\theta \in \Theta(m)$ and particularly for $\theta=\widehat{\theta}(m)$ and  \eqref{eq:mes} holds.
Thus,  \eqref{eq:cr} yields to
\begin{align*}
\widehat{L}_n(\widehat{\theta}(m^*))-\widehat{L}_n(\widehat{\theta}(m))&=
o_P(\kappa_n)+O_P(1)+n\,A(m)+o_P(n)=O_P(1)+n\,A(m)+o_P(n).
\end{align*}
Finally, when $m\in {\cal M}$ such as $m^* \not \subset m$, we have

\begin{equation*}
\widehat{C}(m)-\widehat{C}(m^*)=2\, n\, A(m)+o_P(n)+O_P(1)+\kappa_n(|m|-|m^*|)\limiteproban +\infty
\end{equation*}

since $\kappa_n=o(n)$, therefore $\mathbb{P}\big (\widehat{C}(m) > \widehat{C}(m^*) \big ) \limiten 1$.\\
Thus we have proved the first and most difficult part of  Theorem (\ref{theo:1}). The next lines show the second part which is about the consistency of $\widehat{\theta}(\widehat{m})$.\\

\noindent  Given $\epsilon>0$, we have :\\
\begin{eqnarray*}
    \P\Big( \|\widehat \theta(\widehat m)-\theta^*\|_{i\in m^*} > \epsilon \Big)&=&\P\Big( \|\widehat \theta(\widehat m)-\theta^*\|_{i\in m^*} > \epsilon|\widehat{m}=m^* \Big)\, \P \big (\widehat m=m^* \big ) \\
    && \hspace{3cm}+ \P\Big( \|\widehat \theta(\widehat m)-\theta^*\|_{i\in m^*} > \epsilon|\widehat{m}\ne m^* \Big)\, \P \big (\widehat m \ne m^* \big ).
\end{eqnarray*}
From the strong consistency of the QMLE (see New version of Theorem 1 of \cite{barW}), the first term of the right hand side of the above equation is asymptotically zero and also the second one under the assumptions of the first part of Theorem \ref{theo:1} which gives $\P \big (\widehat{m}\neq m^*\big )  \limiten 0$.  \\
\end{proof}

\subsection{Proof of Theorem \ref{theo:2}}
\begin{proof}
 For $x=(x_i)_{1\leq i \leq d} \in \R^d$, denote $\displaystyle F_n(x)=\P \Big ( \bigcap _{1\leq i \leq d} \sqrt n \, \big ( \widehat \theta(\widehat m)-\theta^* \big )_i  \leq x_i \Big )$. \\
First, we have:
\begin{eqnarray*}
F_n(x)&=&\P \Big ( \bigcap _{1\leq i \leq d} \sqrt n \, \big ( \widehat \theta(\widehat m)-\theta^* \big )_i  \leq x_i ~\big | ~\widehat m=m^* \Big )\, \P \big (\widehat m=m^* \big ) \\
&& \hspace{3cm}+ \P \Big ( \bigcap _{1\leq i \leq d} \sqrt n \, \big ( \widehat \theta(\widehat m)-\theta^* \big )_i  \leq x_i ~\big | ~\widehat m \neq m^* \Big )\, \P \big (\widehat m \neq m^* \big ).
\end{eqnarray*}
Under the assumptions of Theorem \ref{theo:1}, $\P \big (\widehat{m}=m^*\big )  \limiten 1$ and $\P \big (\widehat{m}\neq m^*\big )  \limiten 0$. Therefore the second term in the right side of the previous equality asymptotically vanishes. For the first term, we can write,
\begin{multline*}
\P \Big ( \bigcap _{1\leq i \leq d} \sqrt n \, \big ( \widehat \theta(\widehat m)-\theta^* \big )_i  \leq x_i ~\big | ~\widehat m=m^* \Big ) \\ = \P \Big ( \Big \{  \bigcap _{i\in m^*} \sqrt n \, \big ( \widehat \theta( m^*)-\theta^* \big )_i  \leq x_i \Big \} \, \bigcap ~\Big \{ \bigcap _{i\notin m^*} \sqrt n \, \big ( \widehat \theta( m^*)-\theta^* \big )_i  \leq x_i \Big \} \Big ).
\end{multline*}
Since $\theta(m^*) \in \Theta(m^*)$, $\big (\big (\widehat \theta(m^*)\big )_i\big )_{i \notin m^*} =\big (\theta^*_i \big )_{i \notin m^*}=0$, for $(x_i)_{i \notin m^*}$ a family of non negative real numbers we have:
\begin{multline*}
 \P \Big ( \Big \{  \bigcap _{i\in m^*} \sqrt n \, \big ( \widehat \theta( m^*)-\theta^* \big )_i  \leq x_i \Big \} \, \bigcap ~\Big \{ \bigcap _{i\notin m^*} \sqrt n \, \big ( \widehat \theta( m^*)-\theta^* \big )_i  \leq x_i \Big \} \Big ) \\
\hspace{-3cm}= \P \Big (  \bigcap _{i\in m^*} \sqrt n \, \big ( \widehat \theta( m^*)-\theta^* \big )_i  \leq x_i \Big ) \\
\limiten \P \Big ( \big (F(\theta^*,m^*)^{-1} G(\theta^*,m^*)
F(\theta^*,m^*)^{-1}\big )^{-1/2} Z \leq (x_i)_{i \in m^*}\Big ),
\end{multline*}
with $Z$ a standard Gaussian random vector in $\R^{|m^*|}$ from the central limit theorem \eqref{tlcqmle}, and this achieves the proof of \ref{eq:con2} of Theorem \ref{theo:2}.
\end{proof}

\subsection{Proof of Theorem \ref{theo:3}}
Consider the following notation: for $\theta \in \Theta$ and $m\in {\cal M}$, denote the residuals and quasi-residuals by:
\begin{equation*}
\left \{ \begin{array}{lclclcl}
e_t(\theta)&:=&\displaystyle \big (M_\theta^t\big )^{-1} \big (X_t-f_{\theta}^t\big ) & \mbox{and} & \widehat e_t(\theta)&:=&\displaystyle \big (\widehat M_\theta^t\big )^{-1} \big ( X_t-\widehat f_{\theta}^t \big )\\
e_t(m)&:=&\displaystyle \big (M_{\widehat \theta(m)}^t\big )^{-1} \big ( X_t-f_{\widehat \theta(m)}^t \big )& \mbox{and} & \displaystyle \widehat e_t(m)&:=&\displaystyle \big (M_{\widehat \theta(m)}^t\big )^{-1} \big ( X_t-\widehat f_{\widehat \theta(m)}^t \big )
\end{array}
\right . .
\end{equation*}
For $k\in \{0,1,\ldots,n-1\}$,  $\theta \in \Theta$ and $m\in {\cal M}$, define also the adjusted lag-$k$  covariograms and correlograms of the squared (standardized) residual by:
\begin{equation*}\label{eq:ref1}
\left \{ \begin{array}{ccccccc}
\gamma_k(\theta)&\hspace{-3mm}:=&\hspace{-3mm}\displaystyle \frac 1 n \, \sum_{t=1}^{n-k} \big (e_t^2(\theta)-1 \big )\big ( e_{t+k}^2 (\theta) -1 \big ) &\hspace{-1mm} \mbox{and} & \hspace{-1mm}\widehat \gamma_k(\theta)&\hspace{-3mm}:=&\hspace{-3mm}\displaystyle \frac 1 n \, \sum_{t=1}^{n-k} \big (\widehat e_t^2(\theta) -1 \big )\big (\widehat e_{t+k}^2(\theta) -1 \big )  \\
\gamma_k(m)&\hspace{-3mm}:=&\hspace{-3mm}\displaystyle \frac 1 n \, \sum_{t=1}^{n-k} \big ( e_t^2(m) -1 \big )\big (e_{t+k}^2(m) -1 \big ) &\hspace{-1mm} \mbox{and} & \hspace{-1mm}\widehat \gamma_k(m)&\hspace{-3mm}:=&\hspace{-3mm}\displaystyle \frac 1 n \, \sum_{t=1}^{n-k} \big (\widehat e_t^2(m)-1 \big )\big (\widehat e_{t+k}^2(m) -1 \big )
\end{array}
\right .
\end{equation*}
and $
\rho_k(\theta):=\displaystyle \frac {\gamma_k(\theta)}{\gamma_0(\theta)}, ~ \widehat \rho_k(\theta):=\displaystyle \frac {\widehat \gamma_k(\theta)}{\widehat \gamma_0(\theta)}, ~
\rho_k(m):=\displaystyle \frac {\gamma_k(m)}{\gamma_0(m)}$ and $\widehat \rho_k(m):=\displaystyle \frac {\widehat \gamma_k(m)}{\widehat \gamma_0(m)}$.

 \medskip

\noindent Finally, for $K$ a positive integer, denote the vector of adjusted correlogram:
\begin{equation*}
\widehat{\rho}(\theta):=\big ( \widehat\rho_1(\theta), \ldots,\widehat\rho_{K}(\theta) \big )' \quad\mbox{and}\quad \widehat{\rho}(m):=\big ( \widehat\rho_1(m), \ldots,\widehat\rho_{K}(m) \big )'.
\end{equation*}
\begin{proof}
\hspace{0.5cm} (1) This proof is divided into two parts. In (i) we prove a result that ensures that the asymptotic distributions of the vectors $\widehat{\rho}(\theta)$ and $\rho(\theta)$ are the same. In (ii) we show that the large sample distribution of $\sqrt{n}\rho(m^*)$ is normal with a covariance matrix $V(\theta^*,m^*)$ . Those two conditions do lead well to the asymptotic normality \eqref{eq:por}. \\

\hspace{0.5cm} (i) In this part, we first show that for any $k\in \N$,
\begin{equation}
\sqrt{n}\,  \big \| \widehat{\gamma}_k (\theta)-\gamma_k (\theta) \big \| _{\Theta} \limitesur 0.
\label{eq:dol}
\end{equation}
We have:
\begin{eqnarray*}
\sqrt{n} \big (\widehat{\gamma}_k (\theta)-\gamma_k (\theta) \big )&=& \displaystyle \frac 1 {\sqrt{n}} \,\sum_{t=k+1}^{n} \big(\widehat{e}_t^2(\theta)-1\big)\big(\widehat{e}_{t-k}^2(\theta)-1\big)-\frac 1 {\sqrt{n}} \,\sum_{t=k+1}^{n} \big( e_t^2(\theta)-1\big)\big( e_{t-k}^2(\theta)-1\big)\\
&=& \frac 1 {\sqrt{n}} \,\sum_{t=k+1}^{n} \big (\widehat{e}_t^2(\theta)\widehat{e}_{t-k}^2(\theta)-e_t^2(\theta) e_{t-k}^2(\theta)\big )+\frac 1 {\sqrt{n}} \,\sum_{t=k+1}^{n} \big (\widehat{e}_t^2(\theta)-e_t^2(\theta) \big )\\
&& \hspace{6cm} +\frac 1 {\sqrt{n}} \,\sum_{t=k+1}^{n} \big (e_{t-k}^2(\theta)  -\widehat{e}_{t-k}^2(\theta)\big)\\
&=:& I_1 + I_2 + I_3.
\end{eqnarray*}
Now, we show that $\|I_1\|_\Theta \limiteasn 0$. We can rewrite $I_1$ as follows
\begin{align*}
I_1& =\frac 1 {\sqrt{n}} \,\sum_{t=k+1}^{n} \widehat{e}_{t-k}^2(\theta) \big (\widehat{e}_{t}^2(\theta)-e_{t}^2(\theta) \big )+\frac 1 {\sqrt{n}} \,\sum_{t=k+1}^{n}e_{t}^2(\theta)\big (\widehat{e}_{t-k}^2(\theta)-e_{t-k}^2(\theta) \big )\\
&= \frac 1 {\sqrt{n}} \,\sum_{t=k+1}^{n} \big(\widehat{e}_{t-k}^2(\theta)-e_{t-k}^2(\theta) \big)\big (\widehat{e}_{t}^2(\theta)-e_{t}^2(\theta)\big)+\frac 1 {\sqrt{n}} \, \sum_{t=k+1}^{n}e_{t-k}^2(\theta)\big (\widehat{e}_{t}^2(\theta)-e_{t}^2(\theta) \big) \\
& \hspace{8cm}+\frac 1 {\sqrt{n}} \,\sum_{t=k+1}^{n} e_t^2(\theta)\big (\widehat{e}_{t-k}^2(\theta)-e_{t-k}^2(\theta) \big)\\
&:=I_1^1 +I_1^2+ I_1^3.
\end{align*}
Let us show that $\|I_1^1\|_\Theta \limiteasn 0$ in our two frameworks. \\
a/ If $X\subset AC(M_{\theta},f_{\theta})$, by  Hölder's inequality, it follows from \eqref{e:fond} that,
\begin{eqnarray*}
\E\Big[\Big \|\big (\widehat{e}_{t-k}^2(\theta)-e_{t-k}^2(\theta)\big )\big (\widehat{e}_{t}^2(\theta)-e_{t}^2(\theta)\big ) \Big \|_{\Theta}^{1/2}\Big]
& \le &\Big(\E \big [\big \|\widehat{e}_{t}^2(\theta)-e_{t}^2(\theta)\big \|_{\Theta} \big]\times \E \big [ \big \|\widehat{e}_{t-k}^2(\theta)-e_{t-k}^2(\theta)\big \|_{\Theta} \big ]\Big)^{1/2}.
\end{eqnarray*}
But we have
\[
\big \|\widehat{e}_{t}^2(\theta)-e_{t}^2(\theta)\big \|_{\Theta}  \leq \frac 1 {\underline h} \,\big ( 2|X_t|+\|\widehat f_{\theta}^t \|_\Theta+\| f_{\theta}^t \|_\Theta \big ) \big \|\widehat f_{\theta}^t-f_{\theta}^t \|_\Theta+\frac 4 {\underline h^{3/2}} \,\big ( |X_t|^2+\| f_{\theta}^t \|^2_\Theta \big ) \big \|\widehat M_{\theta}^t-M_{\theta}^t \|_\Theta .
\]
Therefore,
\begin{eqnarray*}
\E \big [ \big \|\widehat{e}_{t}^2(\theta)-e_{t}^2(\theta)\big \|_{\Theta} \big ] & \leq& C \,\Big ( \E \big [ \big (|X_t|^2+\|\widehat f_{\theta}^t \|^2_\Theta+\| f_{\theta}^t \|^2_\Theta \big )\big ]\,\times \,  \E \big [ \big \|\widehat f_{\theta}^t-f_{\theta}^t \|^2_\Theta \big ] \Big )^{1/2} \\
&& \hspace{2cm}+ C \,\Big ( \E \big [ \big ( |X_t|^4+\| f_{\theta}^t \|^2_\Theta \big )\big ] \, \times \, \E \big [ \big \|\widehat M_{\theta}^t-M_{\theta}^t \|^2_\Theta \big ] \Big )^{1/2} \\
& \leq &  C \,\Big (  \E \Big [ \Big | \sum_{j \ge t} \alpha_j (f_\theta,\Theta)X_{t-j}\big |^2 \Big ] \Big )^{1/2}+ C \,\Big (\E \Big [ \Big | \sum_{j \ge t} \alpha_j (M_\theta,\Theta)X_{t-j}\big |^2 \Big ]  \Big )^{1/2} \\
& \leq &  C \,\sum_{j \ge t} \alpha_j (f_\theta,\Theta)+\alpha_j (M_\theta,\Theta),
\end{eqnarray*}
using $\E \big [|X_t|^4+\| f_{\theta}^t \|^2_\Theta+\| \widehat f_{\theta}^t \|^2_\Theta\big ] <\infty$ and Cauchy-Schwarz Inequality. Hence,
\begin{eqnarray*}
\E\Big[\Big \|\big (\widehat{e}_{t-k}^2(\theta)-e_{t-k}^2(\theta)\big )\big (\widehat{e}_{t}^2(\theta)-e_{t}^2(\theta)\big ) \Big \|_{\Theta}^{1/2}\Big]  &\leq &  C \,\sum_{j \ge t-k} \alpha_j (f_\theta,\Theta)+\alpha_j (M_\theta,\Theta).
\end{eqnarray*}
Therefore, from \cite{kounias}, $\| I^1_1\| _{\Theta} \limiteasn 0$ when
\begin{equation}\label{koun2}
\sum_{t=1}^\infty t ^{-1/4} \sum_{j \ge t} \alpha_j (f_\theta,\Theta)+\alpha_j (M_\theta,\Theta) <\infty.
\end{equation}
b/ if $X \subset \widetilde {{\mathcal AC}}(\widetilde H_\theta)$, same computations imply $\| I^1_1\| _{\Theta}\limiteasn 0$ when
\begin{equation}\label{koun3}
\sum_{t=1}^\infty t ^{-1/4} \sum_{j \ge t} \alpha_j (\widetilde H_\theta,\Theta) <\infty.
\end{equation}
Since  $\mathbb{E}\big [\|e_t^2(\theta)\|_\Theta \big ]\le 2 \, \underline{h}^{-1}\mathbb{E}\big [X_t^2+\|f_{\theta}^t\|_\Theta^2\big ]<\infty$ and similarly $\mathbb{E}\big [\|\widehat e_t^2(\theta)\|_\Theta \big ]<\infty$, we deduce from the same inequalities as in the first case of $I_1^1$ that $\| I^2_1\| _{\Theta}\limiteasn 0$ and $\| I^3_1\| _{\Theta}\limiteasn 0$ when
\begin{equation}\label{koun4}
\sum_{t=1}^\infty t ^{-1/4} \Big (\sum_{j \ge t} \alpha_j (f_\theta,\Theta)+\alpha_j (M_\theta,\Theta) + \alpha_j (\widetilde H_\theta,\Theta) \Big )^{1/2}<\infty,
\end{equation}
which is also the condition for insuring that $\|I_2\|_\Theta \limiteasn 0$ and $\|I_3\|_\Theta\limiteasn 0$. This ends the proof of \eqref{eq:dol}.\\
Finally, since $\widehat{\rho}_k(\theta)=\widehat{\gamma}_k(\theta) /\widehat{\gamma}_0(\theta)$ and $\rho_k(\theta)=\gamma_k(\theta)/\gamma_0(\theta)$, with $\gamma_0(\theta)>0$, we deduce under condition \eqref{koun4} that
\begin{equation}\label{corelas}
\sqrt{n}\big \| \widehat{\rho}_k (\theta)-\rho_k (\theta) \big \|_\Theta \limiteasn 0\quad\mbox{for any $k\geq 1$}.
\end{equation}
This also implies
\begin{equation}\label{corelas2}
\sqrt{n}\big | \widehat{\rho}_k (m^*)-\rho_k (m^*) \big | \limiteasn 0\quad\mbox{for any $k\geq 1$}.
\end{equation}
\hspace{0.5cm} (ii) The proof of this result has already been done in \cite{li2} but in a Gaussian framework. We recall here the main lines while avoiding the Gaussian assumption. The first step is to use a Taylor expansion of the function $\gamma$. Hence, we have for each $k=1,\ldots,K$,
\begin{equation}\label{taylor1}
\sqrt n \, \gamma_k(m^*)=\sqrt n \, \gamma_k(\widehat{\theta}(m^*))=\sqrt n \, \gamma_k(\theta^*)+ \partial_\theta \gamma_k(\overline \theta^{(k)}) \sqrt n \,\big ((\widehat{\theta}(m^*))_i-\theta_i^*\big )_{i\in m^*} ,
\end{equation}
\mbox{where} $\partial_\theta \gamma_k={}^t\big (\partial \gamma_k/ \partial \theta_i\big )_{i\in m^*}$, and $\overline \theta^{(k)}$ is in the ball of centre $\theta^*$ and radius $\|(\widehat{\theta}(m^*)-\theta^*)_{i\in m^*}\|$.  We also have
\begin{multline} \label{partial}
\partial_\theta \gamma_k(\theta)=- \frac 2 n \Big (  \sum_{t=k+1}^{n} e_t^2(\theta)\, \big(e_{t-k}^2(\theta) -1\big)\frac{\partial_\theta M_{\theta}^t}{M_{\theta}^t}+e_t(\theta) \big(e_{t-k}^2(\theta) -1\big)\, \frac{\partial_\theta f_{\theta}^t}{M_{\theta}^t}\\
+ e_{t-k}(\theta) \,\big( e_t^2(\theta)-1\big) \frac{\partial_\theta f_{\theta}^{t-k}}{M_{\theta}^{t-k}}+ e_{t-k}^2(\theta)\, \big(e_{t}^2(\theta) -1\big)\frac{\partial_\theta M_{\theta}^{t-k}}{M_{\theta}^{t-k}}\Big ).
\end{multline}
We have $\E \big [ e_{t-k}(\theta^*) \,\big( e_t^2(\theta^*)-1\big) \frac{\partial f_{\theta^*}^{t-k}}{M_{\theta^*}^{t-k}}~| ~\sigma\big ((\xi_s)_{s\leq t-k} \big )\big ]=e_{t-k}(\theta^*)\frac{\partial f_{\theta^*}^{t-k}}{M_{\theta^*}^{t-k}} \E \big [e_t^2(\theta^*)-1\big ]=0$ since we have  assumed $\E[\xi_0^2]=1$.
Moreover, $\E \big [e_t(\theta^*) \frac{\partial f_{\theta^*}^t}{M_{\theta^*}^t}\big ]=\E \big [\xi_t \, \frac{\partial f_{\theta^*}^t}{M_{\theta^*}^t}\big ]=0$ and this implies $\E \big [ e_t(\theta^*) \big(e_{t-k}^2(\theta^*) -1\big)\, \frac{\partial f_{\theta^*}^t}{M_{\theta^*}^t}\big ]=0$. As a consequence, the expectation of the three  last terms of \eqref{partial} vanishes for $\theta=\theta^*$. By using the Ergodic Theorem, we finally obtained:
\[
\partial_\theta \gamma_k(\theta^*) \limiteasn -2 \, \E\Big [ e_k^2(\theta^*)\, \big(e_{0}^2(\theta^*) -1\big)\frac{\partial_\theta M_{\theta^*}^k}{M_{\theta^*}^k}\Big ] =-2 \, \E\Big [ \big(\xi_{0}^2 -1\big)\, \partial_\theta \log \big (M_{\theta^*}^k\big )\Big ].
\]
Moreover, since $\partial^2_{\theta^2}f_\theta$ and $\partial^2_{\theta^2}M_\theta$ exist, and since $\widehat \theta(m^*) \limiteasn \theta^*$, we deduce that the same almost sure convergence occurs for $\partial_\theta \gamma_k(\overline \theta^{(k)})$. Then, we finally obtain
\begin{equation}\label{almost2}
\big (\partial_\theta \gamma_k(\overline \theta^{(k)}) \big )_{1\leq k\leq K} \limiteasn J_K(m^*)=-2 \, \Big ( \E\Big [ \big(\xi_{0}^2 -1\big)\,  \frac {\partial}{\partial \theta_j} \log \big (M_{\theta^*}^i\big )\Big ] \Big )_{1\leq i\leq K, \, j\in m^*}.
\end{equation}
We also established a central limit theorem for $\widehat{\theta}(m^*)$ in \eqref{tlcqmle}, and this implies
\begin{multline}\label{loigamma}
\big (\partial_\theta \gamma_k(\overline \theta^{(k)}) \big )_{1\leq k\leq K}  \sqrt n \,\big ((\widehat{\theta}(m^*))_i-\theta_i^*\big )_{i\in m^*}\\
\limiteloin {\cal
N}_{K}\Big (0 \ , \ J_K(m^*) \, F(\theta^*,m^*)^{-1} G(\theta^*,m^*)
F(\theta^*,m^*)^{-1} J'_K(m^*)\Big ).
\end{multline}
On the other hand, when $\theta=\theta^*$, $e^2_t(\theta^*)=\xi^2_t$ for any $t\in \Z$ and since $\E[\xi_0^2]=1$, we deduce that $\big (e^2_t(\theta^*)-1 \big)_t$ is a sequence of centred iid random variables with variance $\mu_4-1$ with $\mu_4=\E[\xi_0^4]$. In such as case, the asymptotic behavior of the covariograms is well known and we deduce:
\begin{equation}\label{covargamma}
\sqrt n \, \big (\gamma_k(\theta^*) \big )_{1\leq k\le K} \limiteloin {\cal N}_K \big (0 \, , \, (\mu_4-1)^2\, I_K \big ),
\end{equation}
with $I_k$ the $(K\times K)$ identity matrix.

 \medskip

\noindent We would like to use \eqref{taylor1} for obtaining the asymptotic behavior of $\gamma(m^*)$. In \eqref{loigamma} and \eqref{covargamma}, we obtained the asymptotic normality of each of the two terms composing $\gamma(m^*)$. Now we need to study the joint asymptotic behavior of $\sqrt{n} \, \gamma(\theta^*)$ and $\sqrt n \,\big ((\widehat{\theta}(m^*))_i-\theta_i^*\big )_{i\in m^*}$.

\medskip

\noindent Using the proof of the asymptotic normality of the QMLE (see for instance \cite{barW}), a Taylor expansion of log-likelihood for large $n$ leads to

\[
\big ((\widehat{\theta}(m^*))_i-\theta_i^*\big )_{i\in m^*}\approx -\big(F(\theta^*,m^*)\big)^{-1} \frac 1 n \, \frac{\partial}{\partial \theta}L_n(\theta^*).
\]
Therefore, the asymptotic cross expectation between $\big (\partial_\theta \gamma_k(\overline \theta^{(k)}) \big )_{ k} \sqrt n \,\big ((\widehat{\theta}(m^*))_i-\theta_i^*\big )_{i\in m^*}$ and $\sqrt{n} \, \gamma(\theta^*)$ is equal to:
\begin{equation}
-\,J_K(m^*) \, F(\theta^*,m^*)^{-1}\E\Big [\frac{\partial}{\partial \theta}L_n(\theta^*)\, \gamma(\theta^*)' \Big ] .
\end{equation}
From \eqref{eq:eq1}, a direct differentiation of $L_n$ provides
\begin{equation*}
\frac{\partial }{\partial \theta}L_n(\theta^*)= \sum_{t=1}^{n} \big(e_t^2(\theta^*)-1 \big) \, \frac{\partial}{\partial \theta} \log \big (M_{\theta^*}^t\big )  + \sum_{t=1}^{n} e_t(\theta^*)\, \frac{\partial }{\partial \theta}f_{\theta^*}^t
\end{equation*}
so that,
\begin{eqnarray*}
\E\Big [\frac{\partial}{\partial \theta}L_n(\theta^*)\, \gamma_k(\theta^*)\Big ]&=&
\frac 1 n \, \E \Big[ \sum_{i=1}^{n} \big(e_i^2(\theta^*)-1 \big) \, \frac{\partial}{\partial \theta} \log \big (M_{\theta^*}^i\big )\sum_{j=k+1}^{n} \big( e_j^2(\theta^*)-1\big)\, \big( e_{j-k}^2(\theta^*)-1\big)\Big]\\
&& \hspace{2cm}+\frac 1 n \, \E \Big[\sum_{i=1}^{n} e_i(\theta^*)\, \frac{\partial }{\partial \theta}f_{\theta^*}^i \sum_{j=k+1}^{n} \big( e_j^2(\theta^*)-1\big)\, \big( e_{j-k}^2(\theta^*)-1\big)\Big] \\
&=&
\frac 1 n \, \sum_{i=1}^{n} \sum_{j=k+1}^{n}\E \Big[  \big(\xi_i^2-1 \big) \, \big( \xi_j^2-1\big)\, \big( \xi_{j-k}^2-1\big)\, \frac{\partial}{\partial \theta} \log \big (M_{\theta^*}^i\big ) \Big]\\
&& \hspace{2cm}+\frac 1 n \,\sum_{i=1}^{n}\sum_{j=k+1}^{n}  \E \Big[ \xi_i\,\big( \xi_j^2-1\big)\, \big( \xi_{j-k}^2-1\big)\,  \frac{\partial }{\partial \theta}f_{\theta^*}^i \Big].
\end{eqnarray*}
Using conditional expectations, we have $\E \Big[  \big(\xi_i^2-1 \big) \, \big( \xi_j^2-1\big)\, \big( \xi_{j-k}^2-1\big)\, \frac{\partial}{\partial \theta} \log \big (M_{\theta^*}^i\big ) \Big]=0$ for $i\neq j$ since $k\geq 1$. Moreover, for $i=j$, we obtain:
\[
\E \Big[  \big(\xi_i^2-1 \big) \, \big( \xi_j^2-1\big)\, \big( \xi_{j-k}^2-1\big)\, \frac{\partial}{\partial \theta} \log \big (M_{\theta^*}^i\big ) \Big]=(\mu_4-1) \, \E \Big [  \big( \xi_{i-k}^2-1\big)\, \frac{\partial}{\partial \theta} \log \big (M_{\theta^*}^i\big ) \Big],
\]
which is the row $k$ of matrix $-\frac{(\mu_4-1)}{2} \, J_K(m^*)$. Similarly, and using the assumption  $\E \big [\xi^3_0]=0$, we obtain $ \E \Big[ \xi_i\,\big( \xi_j^2-1\big)\, \big( \xi_{j-k}^2-1\big)\,  \frac{\partial }{\partial \theta}f_{\theta^*}^i \Big]=0$ for any $i,j$ and $k$. As a consequence,
\begin{multline*}
\cov \big ( \sqrt{n} \, \gamma(\theta^*)\, , \,\big (\partial_\theta \gamma_k(\overline \theta^{(k)}) \big )_{ k} \sqrt n \,\big ((\widehat{\theta}(m^*))_i-\theta_i^*\big )_{i\in m^*}\big ) \\
\limiten \frac 1 2 \, (\mu_4-1) \,J_K(m^*) \, F(\theta^*,m^*)^{-1}\, J'_K(m^*).
\end{multline*}
Finally, we deduce the asymptotic covariance matrix of $\sqrt{n}\, \gamma(m^*)$, which is
\begin{multline*}
(\mu_4-1)^2 \, I_{K}+J_K(m^*) \, F(\theta^*,m^*)^{-1} G(\theta^*,m^*)
F(\theta^*,m^*)^{-1} J'_K(m^*) \\+  (\mu_4-1) \,J_K(m^*) \, F(\theta^*,m^*)^{-1}\, J'_K(m^*).
\end{multline*}
Moreover the vector $\gamma(m^*)$ is normal distributed from Lemma 3.3 of \cite{Ling1997}.\\
\noindent Thus, using Slutsky Lemma and with $\gamma_0(m^*) \limiteasn \mu_4-1$, and with $\rho_k(m^*)=\gamma_k(m^*)/\gamma_0(m^*) $, the limit theorem \eqref{eq:por} holds with
\begin{multline}\label{VV}
V(\theta^*, m^*):=I_{K}+(\mu_4-1)^{-2} \, J_K(m^*) \, F(\theta^*,m^*)^{-1} G(\theta^*,m^*)
F(\theta^*,m^*)^{-1} J'_K(m^*) \\+  (\mu_4-1)^{-1} \, J_K(m^*) \, F(\theta^*,m^*)^{-1}\, J'_K(m^*).
\end{multline}
The proof is achieved after using the limit theorem \eqref{corelas2}. \\
~\\

(2) (\ref{eq:chi}) follows directly from (\ref{eq:por}).
\\

 (3) We follow a same reasoning like in the proof of Theorem \ref{theo:2}.
For $x=(x_k)_{1\leq k \leq K} \in \R^K$, denote by $\displaystyle F_n(x)=\P \Big ( \bigcap _{1\leq k \leq K} \sqrt n \, \big ( \widehat \rho(\widehat m) \big )_k  \leq x_k \Big )$ the distribution function of $\sqrt n \widehat \rho(\widehat m).$ \\
Applying the Total Probability Rule and by virtue of Theorem \ref{theo:1}, we obtain:
\[F_n(x)= \P \Big ( \bigcap _{1\leq k \leq K} \sqrt n \, \big ( \widehat \rho(m^*) \big )_k  \leq x_k \Big ). \]
Therefore, the vectors $\sqrt n \widehat \rho(\widehat m) $ and $\sqrt n \widehat \rho(m^*) $ have exactly the same distribution.
\end{proof}

\bibliography{biblio5}

\end{document}